\documentclass[12pt,leqno]{article}
\usepackage{amsmath,amssymb,bbm,array}
\usepackage{amsthm}
\usepackage{titlesec} 
\usepackage[all,cmtip]{xy}

\titlespacing{\section}{0cm}{3.5pc}{1.5pc}

\makeatletter
\def\@cite#1{#1}
\def\@citex[#1]#2{\if@filesw\immediate\write\@auxout{\string\citation{#2}}\fi
  \def\@citea{}\@cite{\@for\@citeb:=#2\do
    {\@citea\def\@citea{\@citesep}\@ifundefined
       {b@\@citeb}{{\bf ?}\@warning
       {Citation `\@citeb' on page \thepage \space undefined}}%
{\csname b@\@citeb\endcsname}}}{#1}}
\def\@citesep{; }
\makeatother

\newtheoremstyle{Kang}{}{}{\itshape}{}{\bf}{}{.5em}{}
\theoremstyle{Kang}
\newtheorem{theorem}{Theorem}[section]

\newtheorem{lemma}[theorem]{Lemma}
\newtheorem{prop}[theorem]{Proposition}

\newtheoremstyle{Kremark}{}{}{}{}{\bf}{}{.5em}{}
\theoremstyle{Kremark}
\newtheorem{defn}[theorem]{Definition}
\newtheorem{example}[theorem]{Example}
\newtheorem*{remark}{Remark.}
\newtheorem{other}{}
\newenvironment{idef}[1]{\renewcommand{\theother}{#1}\begin{other}}{\end{other}}
\newenvironment{Case}[1]{\medskip {\it Case #1.}}{}

\allowdisplaybreaks[1]  
\numberwithin{equation}{section}

\def\fn#1{\mathop{{\rm #1}\vphantom{\sin}}} 
\def\bm#1{\mathbbm{#1}}

\def\ul#1{\mathop{\underline{#1}}}
\def\ol#1{\mathop{\overline{#1}}}
\def\q{\quad\;}  

\title{Quasi-monomial actions and some 4-dimensional rationality problems}
\author{Akinari Hoshi$^{1}$, Ming-chang Kang$^{2}$ and Hidetaka Kitayama$^{3}$ \\[3mm]
\begin{minipage}{16cm} \begin{description} \itemsep=-1pt
\item[] $^{(1)}$Department of Mathematics, Rikkyo University,
Tokyo, Japan \item[] $^{(2)}$Department of Mathematics and Taida
Institute of Mathematical\\ Sciences, National Taiwan University,
Taipei, Taiwan \item[] $^{(3)}$Department of Mathematics, Osaka
University, Osaka, Japan
\end{description} \end{minipage}}
\date{}

\begin{document}

\maketitle

\footnote{\textit{\!\!\!$2010$ Mathematics Subject
Classification}. Primary 12F20, 13A50, 14E08.}
\footnote{\textit{\!\!\!Keywords and phrases}. Rationality
problem, monomial actions, embedding problem, Noether's problem,
algebraic tori.}\footnote{E-mails : hoshi@rikkyo.ac.jp,
kang@math.ntu.edu.tw, h-kitayama@cr.math.sci.osaka-u.ac.jp}
\footnote{\!\!\!The first-named author was partially supported by
KAKENHI (22740028). The second-named author was partially
supported by the National Center for Theoretic Sciences (Taipei
Office). Some part of this work was done during the first-named
author and the third-named author visited the National Center for
Theoretic Sciences (Taipei Office), whose support is gratefully
acknowledged.}

\begin{abstract}
{\bf Abstract.} Let $G$ be a finite group acting on
$k(x_1,\ldots,x_n)$, the rational function field of $n$ variables
over a field $k$. The action is called a purely monomial action if
$\sigma\cdot x_j=\prod_{1\le i\le n} x_i^{a_{ij}}$ for all $\sigma
\in G$, for $1\le j\le n$ where $(a_{ij})_{1\le i,j\le n} \in
GL_n(\bm{Z})$. The main question is that, under what situations,
the fixed field $k(x_1,\ldots,x_n)^G$ is rational (= purely
transcendental) over $k$. This rationality problem has been
studied by Hajja, Kang, Hoshi, Rikuna when $n\le 3$. In this paper
we will prove that $k(x_1,x_2,x_3,x_4)^G$ is rational over $k$
provided that the purely monomial action is decomposable. To prove
this result, we introduce a new notion, the quasi-monomial action,
which is a generalization of previous notions of multiplicative
group actions. Moreover, we determine the rationality problem of
purely quasi-monomial actions of $K(x, y)^G$ over $k$ where $k=
K^G$.
\end{abstract}

\clearpage \leftmargini=5ex
\section{Introduction}

Let $k$ be a field, $L$ be a finitely generated field extension of
$k$. $L$ is called $k$-rational (or rational over $k$) if $L$ is
purely transcendental over $k$, i.e.\ $L$ is isomorphic to $k(x_1,
\ldots, x_n)$, the rational function field of $n$ variables over
$k$ for some integer $n$. $L$ is called $k$-unirational (or
unirational over $k$) if $k \subset L \subset k(x_1, \ldots, x_n)$
for some integer $n$. It is clear that
``$k$-rational"$\Rightarrow$``$k$-unirational". The classical
rationality problem (also known as the L\"uroth problem) asks
whether a $k$-unirational field is $k$-rational. Although the
rationality problem has counter-examples in dimension 3 in 1970's,
many special cases of it, e.g. Noether's problem for non-abelian
groups, are still mysterious and await to be explored. The reader
is referred to the papers [\cite{MT,CTS,Sw}] for surveys of the
various rationality problems and Noether's problem.

The rationality problem of twisted multiplicative actions on
rational function fields [\cite{Sa1}, p.538; \cite{Sa2}, p.535;
\cite{Ka4}, Section 2] is ubiquitous in the investigation of
Noether's problem and other rationality problems. In this paper,
we introduce a more general twisted multiplicative action, the
quasi-monomial action, which arose already in the literature
[\cite{Sa1}, p.542; \cite{Ka3}, p.2773]. A good understanding of
the rationality problem of quasi-monomial actions is useful to
solving other rationality problems.

\begin{defn} \label{d1.1}
Let $G$ be a finite subgroup of $\fn{Aut}_k(K(x_1,\ldots,x_n))$ where $K/k$ is a finite field extension and
$K(x_1,\ldots,x_n)$ is the rational function field of $n$ variables over $K$.
The action of $G$ on $K(x_1,\ldots,x_n)$ is called a quasi-monomial action if it satisfies the following three conditions:
\begin{enumerate}
\item[(i)] $\sigma(K)\subset K$ for any $\sigma\in G$; \item[(ii)]
$K^G=k$, where $K^G$ is the fixed field under the action of $G$;
and \item[(iii)] for any $\sigma\in G$ and $1 \le j \le n$,
$\sigma(x_j)=c_j(\sigma)\prod_{i=1}^n x_i^{a_{ij}}$ where
$c_j(\sigma)\in K^\times$ and $[a_{i,j}]_{1\le i,j \le n} \in
GL_n(\bm{Z})$.
\end{enumerate}

The quasi-monomial action is called a purely quasi-monomial action if $c_j(\sigma)=1$ for any $\sigma \in G$, any $1\le j\le n$ in (iii).

To simplify the wordings in the above situation, we will say that
$G$ acts on $K(x_1,\ldots,$ $x_n)$ by purely quasi-monomial
$k$-automorphisms or quasi-monomial $k$-automorphisms depending on
the coefficients $c_j(\sigma)=1$ or not (thus the assumption
$k=K^G$ is understood throughout this paper).

Note that the possibility $k=K$ is not excluded. In fact, when
$k=K$, it becomes monomial actions or purely monomial actions
which are prototypes of quasi-monomial actions (see Example
\ref{ex1.3}).
\end{defn}

The main question of this paper is:

\begin{idef}{Question.}
Under what situation the fixed field $K(x_1,\ldots,x_n)^G$ is
$k$-rational if $G$ acts on $K(x_1,\ldots,x_n)$ by quasi-monomial
$k$-automorphisms ?
\end{idef}

\begin{example} \label{ex1.2}
When $G$ acts by purely quasi-monomial $k$-automorphisms and
$G\simeq \fn{Gal}(K/k)$, i.e.\ $G$ acts faithfully on $K$, the
fixed field $K(x_1,\ldots,x_n)^G$ is a function field of some
algebraic torus defined over $k$ and split over $K$ [\cite{Vo2}].
It is easy to see that all the 1-dimensional algebraic tori are
rational. The birational classification of the 2-dimensional
algebraic tori and the 3-dimensional algebraic tori has been
studied by Voskresenskii [\cite{Vo1}] and Kunyavskii [\cite{Ku}]
respectively. We record their results as the following two
theorems.
\end{example}

\begin{theorem}[{Voskresenskii [\cite{Vo1}]}] \label{t1.11}
Let $k$ be a field. Then all the two-dimensional algebraic
$k$-tori are rational over $k$. In particular, $K(x_1,x_2)^G$ is
always $k$-rational if $G$ is faithful on $K$ and $G$ acts on
$K(x_1,x_2)$ by purely monomial $k$-automorphisms.
\end{theorem}

\begin{theorem}[{Kunyavskii [\cite{Ku}; \cite{Ka4}, Section 1]}] \label{t1.12}
Let $k$ be a field. Then all the three-dimensional algebraic
$k$-tori are rational over $k$ except for the $15$ cases in the
list of \rm{[\cite{Ku}, Theorem 1]}. For the exceptional $15$
cases, they are not rational over $k$; in fact, they are even not
retract rational over $k$.
\end{theorem}

For the definition of retract rationality, see Definition
\ref{d6.1}.

\begin{example} \label{ex1.3}
When $G$ acts trivially on $K$, i.e.\ $k=K$, the quasi-monomial
action is called the monomial action. When $k=K$ and the
quasi-monomial action is purely, it is called a purely monomial
action. Monomial actions and purely monomial actions on
$k(x_1,x_2)$ and $k(x_1,x_2,x_3)$ were studied by Hajja, Kang,
Hoshi, Rikuna, Prokhorov, Yamasaki, Kitayama, etc.. But almost
nothing is known about the rationality of $k(x_1,x_2,x_3,x_4)^G$
where $G$ acts by purely monomial $k$-automorphisms, to say
nothing for the monomial $k$-actions. In the following we list the
known results for the monomial actions and purely monomial actions
on $k(x_1,x_2)$ and $k(x_1,x_2,x_3)$.
\end{example}

\begin{theorem} [{Hajja [\cite{Ha}]}] \label{t1.13}
Let $k$ be a field, $G$ be a finite group acting on $k(x_1,x_2)$
by monomial $k$-automorphisms. Then $k(x_1,x_2)^{G}$ is rational
over $k$.
\end{theorem}

\begin{theorem} [{Hajja, Kang, Hoshi, Rikuna [\cite{HK1,HK2,HR}]}] \label{t1.14}
Let $k$ be a field, $G$ be a finite group acting on
$k(x_1,x_2,x_3)$ by purely monomial $k$-automorphisms. Then
$k(x_1,x_2,x_3)^G$ is rational over $k$.
\end{theorem}

\begin{theorem} [{Yamasaki, Hoshi, Kitayama, Prokhorov [\cite{Ya,HKY,Pr}]}] \label{t1.15}
Let $k$ be a field of {\rm char} $k\neq 2$, $G$ be a finite group
acting on $k(x_1,x_2,x_3)$ by monomial $k$-automorphisms such that
$\rho_{\ul{x}} : G\to GL_3 (\bm{Z})$ is injective \rm{(}where
$\rho_{\ul{x}}$ is the group homomorphism in Definition
\ref{d1.5}\rm{)}. Then $k(x_1,x_2,x_3)^G$ is rational over $k$
except for the 8 cases contained in \rm{[\cite{Ya}]} and one
additional case. For the last exceptional case, the fixed field
$k(x_1,x_2,x_3)^G$ is also rational over $k$ except for a minor
situation. In particular, if $k$ is an algebraically closed field
with {\rm char} $k \neq 2$, then $k(x_1,x_2,x_3)^G$ is rational
over $k$ for any action of $G$ on $k(x_1,x_2,x_3)$ by monomial
$k$-automorphisms.
\end{theorem}

\begin{theorem} [{Kang, Michailov, Zhou [\cite{KMZ}, Theorem 2.5]}] \label{t1.16}
Let $k$ be a field with $\fn{char}k\ne 2$, $\alpha$ be a non-zero
element in $k$. Let $G$ be a finite group acting on
$k(x_1,x_2,x_3)$ by monomial $k$-automorphisms with all the
coefficients $c_j(\sigma)$ \rm{(}in Definition \ref{d1.1}\rm{)}
belonging to $\langle \alpha \rangle$. Assume that $\sqrt{-1}$ is
in $k$. Then $k(x_1,x_2,x_3)^G$ is rational over $k$.
\end{theorem}

\begin{example} \label{ex1.4}
Assume that $G\simeq \fn{Gal}(K/k)$, i.e.\ $G$ acts faithfully on
$K$. Suppose that $a_\sigma \in GL_{n+1}(K)$ for each $\sigma \in
G$. Denote by $\bar{a}_\sigma$ the image of $a_\sigma$ in the
canonical map $GL_{n+1}(K)\to PGL_{n+1}(K)$. Consider the rational
function fields $K(y_0,y_1,\ldots,y_n)$ and $K(x_1$, $\ldots,x_n)$
where $x_i=y_i/y_0$ for $1\le i\le n$. For each $\sigma\in G$,
$a_\sigma$ induces a $k$-automorphism on $K(y_0,y_1,\ldots,y_n)$
and $K(x_1,\ldots,x_n)$ (note that elements of $K$ in
$K(y_0,\ldots,y_n)$ are acted through $\fn{Gal}(K/k)$). Assume
furthermore that the map $\gamma:G\to PGL_n(K)$ defined by
$\gamma(\sigma)=\bar{a}_\sigma$ is a 1-cocycle, i.e.\ $\gamma\in
H^1(G,PGL_n(K))$. Then $G$ induces an action on
$K(x_1,\ldots,x_n)$. The fixed field $K(x_1,\ldots,x_n)^G$ is
called a Brauer-field $F_{n,k}(\gamma)$, i.e.\ the function field
of an $n$-dimensional Severi-Brauer variety over $k$ associated to
$\gamma$ [\cite{Ro1,Ro2,Ka1}]. It is known that a Brauer-field
over $k$ is $k$-rational if and only if it is $k$-unirational
[\cite{Se}, page 160]. If we assume that each $a_\sigma$ is an
$M$-matrix, i.e.\ each column of $a_\sigma$ has precisely one
non-zero entry, then the action of $G$ on $K(x_1,\ldots,x_n)$
becomes a quasi-monomial action.
\end{example}

From Example \ref{ex1.4}, we find that, for a quasi-monomial
action of $G$, the fixed field $K(x_1,\ldots,x_n)^G$ is not always
$k$-unirational. However, for the quasi-monomial actions discussed
in Example \ref{ex1.2} and Example \ref{ex1.3},
$K(x_1,\ldots,x_n)^G$ are always $k$-unirational (see [\cite{Vo2},
page 40, Example 21]).

Before stating the first main result of this paper, we define another terminology related to a quasi-monomial action.

\begin{defn} \label{d1.5}
Let $G$ act on $K(x_1,\ldots,x_n)$ by quasi-monomial
$k$-automorphisms with the conditions (i), (ii), (iii) in
Definition \ref{d1.1}. Define a group homomorphism $\rho_{\ul{x}}
: G\to GL_n (\bm{Z})$ by $\Phi_{\ul{x}}(\sigma)=[a_{i,j}]_{1\le
i,j\le n} \in GL_n(\bm{Z})$ for any $\sigma \in G$ where the
matrix $[a_{i,j}]_{1\le i,j \le n}$ is given by
$\sigma(x_j)=c_j(\sigma)\prod_{1\le i\le n} x_i^{a_{ij}}$ in (iii)
of Definition \ref{d1.1}.
\end{defn}

\begin{prop} \label{p1.6}
Let $G$ be a finite group acting on $K(x_1,\ldots,x_n)$ by
quasi-monomial $k$-automorphisms. Then there is a normal subgroup
$N$ of $G$ satisfying the following conditions \rm{:} {\rm (i)}
$K(x_1,\ldots,x_n)^N=K^N(y_1,\ldots,y_n)$ where each $y_i$ is of
the form $ax_1^{e_1}x_2^{e_2}\cdots x_n^{e_n}$ with $a\in
K^{\times}$ and $e_i\in \bm{Z}$ \rm{(}we may take $a=1$ if the
action is a purely quasi-monomial action\rm{)}, {\rm (ii)} $G/N$
acts on $K^N(y_1,\ldots,y_n)$ by quasi-monomial $k$-automorphisms,
and {\rm (iii)} $\rho_{\ul{y}}: G/N\to GL_n(\bm{Z})$ is an
injective group homomorphism.
\end{prop}

Here is an easy application of Proposition \ref{p1.6}.

\begin{prop} \label{p1.7}
{\rm (1)} Let $G$ be a finite group acting on $K(x)$ by purely quasi-monomial $k$-automorphisms.
Then $K(x)^G$ is $k$-rational.

{\rm (2)} Let $G$ be a finite group acting on $K(x)$ by
quasi-monomial $k$-automorphisms. Then $K(x)^G$ is $k$-rational
except for the following case \rm{:} There is a normal subgroup
$N$ of $G$ such that {\rm (i)} $G/N=\langle \sigma \rangle \simeq
C_2$, {\rm (ii)} $K(x)^N=k(\alpha)(y)$ with $\alpha^2=a\in
K^{\times}$, $\sigma(\alpha)=-\alpha$ \rm{(}if $\fn{char}k \ne
2$\rm{)}, and $\alpha^2+\alpha=a\in K$, $\sigma(\alpha)=\alpha+1$
\rm{(}if $\fn{char} k=2$\rm{)}, {\rm (iii)} $\sigma\cdot y=b/y$
for some $b\in k^{\times}$.

For the exceptional case, $K(x)^G=k(\alpha) (y)^{G/N}$ is
$k$-rational if and only if the norm residue $2$-symbol
$(a,b)_k=0$ (if $\fn{char}k\ne 2$), and $[a,b)_k=0$ (if
$\fn{char}k=2$).

Moreover, if $K(x)^G$ is not $k$-rational, then $k$ is an infinite
field, the Brauer group $\fn{Br}(k)$ is non-trivial, and $K(x)^G$
is not $k$-unirational.
\end{prop}

For the definition of the norm residue $2$-symbols $(a,b)_k$ and
$[a,b)_k$, see [\cite{Dr}, Chapter 11].

The first main result of this paper is the following.

\begin{theorem} \label{t1.8}
Let $G$ be a finite group acting on $K(x,y)$ by purely
quasi-monomial $k$-automorphisms. Define $N=\{\sigma\in G:
\sigma(x)=x,~ \sigma(y)=y\}$, $H=\{\sigma\in G:
\sigma(\alpha)=\alpha$ for all $\alpha\in K\}$. Then $K(x,y)^G$ is
rational over $k$ except possibly for the following situation
\rm{:}  \rm{(}1\rm{)} $\fn{char}k\ne 2$ and \rm{(}2\rm{)}
$(G/N,HN/N)\simeq (C_4,C_2)$ or $(D_4,C_2)$.

More precisely, in the exceptional situation we may choose $u,v\in
k(x,y)$ satisfying that $k(x,y)=k(u,v)$ (and therefore
$K(x,y)=K(u,v)$) such that
\begin{enumerate}
\item[{\rm (i)}] when $(G/N,HN/N)\simeq (C_4,C_2)$,
$K^N=k(\sqrt{a})$ for some $a\in k\backslash k^2$, $G/N=\langle
\sigma \rangle \simeq C_4$, then $\sigma$ acts on $K^N(u,v)$ by
$\sigma: \sqrt{a} \mapsto -\sqrt{a}$, $u\mapsto \frac{1}{u}$,
$v\mapsto -\frac{1}{v}$; or \item[{\rm (ii)}] when
$(G/N,HN/N)\simeq (D_4,C_2)$, $K^N=k(\sqrt{a},\sqrt{b})$ is a
biquadratic extension of $k$ with $a,b\in k\backslash k^2$,
$G/N=\langle \sigma,\tau \rangle \simeq D_4$, then $\sigma$ and
$\tau$ act on $K^N(u,v)$ by $\sigma:\sqrt{a} \mapsto -\sqrt{a}$,
$\sqrt{b}\mapsto \sqrt{b}$, $u\mapsto\frac{1}{u}$, $v\mapsto
-\frac{1}{v}$, $\tau: \sqrt{a}\mapsto \sqrt{a}$, $\sqrt{b}\mapsto
-\sqrt{b}$, $u\mapsto u$, $v\mapsto -v$.
\end{enumerate}
For Case {\rm (i)}, $K(x,y)^G$ is $k$-rational if and only if the
norm residue $2$-symbol $(a,-1)_k=0$. For Case {\rm (ii)},
$K(x,y)^G$ is $k$-rational if and only if $(a,-b)_k=0$.

Moreover, if $K(x,y)^G$ is not $k$-rational, then $k$ is an
infinite field, the Brauer group $\fn{Br}(k)$ is non-trivial, and
$K(x,y)^G$ is not $k$-unirational.
\end{theorem}

The following definition gives an equivalent definition of
quasi-monomial actions, which will be used in our second main
result. This definition follows the approach of Saltman's
definition of twisted multiplicative actions [\cite{Sa1,Sa2};
\cite{Ka4}, Definition 2.2].

\begin{defn} \label{d1.9}
Let $G$ be a finite group. A $G$-lattice $M$ is a finitely
generated $\bm{Z}[G]$-module which is $\bm{Z}$-free as an abelian
group, i.e.\ $M=\bigoplus_{1\le i\le n} \bm{Z}\cdot x_i$ with a
$\bm{Z}[G]$-module structure. Let $K/k$ be a field extension such
that $G$ acts on $K$ with $K^G=k$. Consider a short exact sequence
of $\bm{Z}[G]$-modules $\alpha: 1\to K^{\times}\to M_\alpha \to
M\to 0$ where $M$ is a $G$-lattice and $K^{\times}$ is regarded as
a $\bm{Z}[G]$-module through the $G$-action on $K$. The
$\bm{Z}[G]$-module structure (written multiplicatively) of
$M_\alpha$ may be described as follows : For each $x_j\in M$
(where $1\le j\le n$), take a pre-image $u_j$ of $x_j$. As an
abelian group, $M_\alpha$ is the direct product of $K^{\times}$
and $\langle u_1, \ldots, u_n \rangle$. If $\sigma \in G$ and
$\sigma\cdot x_j=\sum_{1\le i\le n} a_{ij} x_i \in M$, we find
that $\sigma\cdot u_j=c_j(\sigma) \cdot \prod_{1\le i\le n}
u_i^{a_{ij}} \in M_\alpha$ for a unique $c_j(\sigma)\in
K^{\times}$ determined by the group extension $\alpha$.

Using the same idea, once a group extension $\alpha:1\to
K^{\times}\to M_\alpha \to M\to 0$ is given, we may define a
quasi-monomial action of $G$ on the rational function field
$K(x_1,\ldots,x_n)$ as follows : If $\sigma\cdot x_j=\sum_{1\le
i\le n} a_{ij} x_i \in M$, then define $\sigma\cdot
x_j=c_j(\sigma)\prod_{1\le i\le n} x_i^{a_{ij}} \in
K(x_1,\ldots,x_n)$ and $\sigma\cdot \alpha =\sigma(\alpha)$ for
$\alpha\in K$ where $\sigma(\alpha)$ is the image of $\alpha$
under $\sigma$ via the prescribed action of $G$ on $K$. This
quasi-monomial action is well-defined (see [\cite{Sa1}, p.538] for
details). The field $K(x_1,\ldots,x_n)$ with such a $G$-action
will be denoted by $K_\alpha(M)$ to emphasize the role of the
extension $\alpha$; its fixed field is denoted as $K_\alpha(M)^G$.
We will say that $G$ acts on $K_\alpha(M)$ by quasi-monomial
$k$-automorphisms.

If $k=K$, then $k_\alpha(M)^G$ is nothing but the fixed field
associated to the monomial action discussed in Example
\ref{ex1.3}.

If the extension $\alpha$ splits, then we may take
$u_1,\ldots,u_n\in M_\alpha$ satisfying that $\sigma\cdot
u_j=\prod_{1\le i\le n} u_i^{a_{ij}}$. Hence the associated
quasi-monomial action of $G$ on $K(x_1,\ldots,x_n)$ becomes a
purely quasi-monomial action. In this case, we will write
$K_\alpha(M)$ and $K_\alpha(M)^G$ as $K(M)$ and $K(M)^G$
respectively (the subscript $\alpha$ is omitted because the
extension $\alpha$ plays no important role). We will say that $G$
acts on $K(M)$ by purely quasi-monomial $k$-automorphisms. Again
$k(M)^G$ is the fixed field associated to the purely monomial
action discussed in Example \ref{ex1.3}.
\end{defn}

Now we arrive at the second main result of this paper. As
mentioned in Example \ref{ex1.3}, the rationality problem of
$k(M)^G$ where $M$ is a $G$-lattice of $\bm{Z}$-rank 4 is still
unsolved. With the aid of Theorem \ref{t1.8}, we are able to show
that $k(M)^G$ is $k$-rational whenever $M$ is a decomposable
$G$-lattice of $\bm{Z}$-rank 4. From the list of [\cite{BBNWZ}],
there are 710 finite subgroups in $GL_4(\bm{Z})$ up to conjugation
(i.e. there are 710 lattices of $\bm{Z}$-rank 4); the total number
of decomposable ones is 415. Here is the precise statement of our
result.

\begin{theorem} \label{t1.10}
Let $k$ be a field, $G$ be a finite group, $M$ be a $G$-lattice
with $\fn{rank}_{\bm{Z}} M=4$ such that $G$ acts on $k(M)$ by
purely monomial $k$-automorphisms. If $M$ is decomposable, i.e.\
$M=M_1\oplus M_2$ as $\bm{Z}[G]$-modules where $1\le
\fn{rank}_{\bm{Z}} M_1 \le 3$, then $k(M)^G$ is $k$-rational.
\end{theorem}

We remark that the analogous question of the above theorem for the
case of algebraic tori (i.e. $G \simeq Gal(K/k)$ and $M$ is a
decomposable $G$-lattice) is not so challenging, because the
rationality of $K(M)^G$ can be reduced to those of $K(M_1)^G$ and
$K(M_2)^G$ (see Theorem \ref{t6.5}). On the other hand, we may
adapt the proof of Theorem \ref{t1.8} to study the rationality of
$k(M)^G$ when $\fn{rank}_{\bm{Z}}M=5$ and $M$ is some special
decomposable $G$-lattice. See Theorem \ref{t6.2}.

We organize this paper as follows. The proof of Propositions
\ref{p1.6} and \ref{p1.7} is given in Section 2. Section 3
contains a list of all finite subgroups of $GL_2(\bm{Z})$ and
three rationality results which will be used in Section 4. Lemma
\ref{l3.3} is of interests itself. The proof of Theorem \ref{t1.8}
is given in Section 4. Section 5 contains the proof of Theorem
\ref{t1.10}. Section 6 is devoted to a 5-dimensional rationality
problem.

\begin{idef}{\indent Standing notations.}
Throughout this paper, $G$ is always a finite group.
$K(x_1,\ldots$, $x_n)$ denotes the rational function field of $n$
variables over a field $K$; when $n=1$, we will write $K(x)$; when
$n=2$, we will write $K(x,y)$. For a field $K$, $K^{\times}$
denote the set of all non-zero elements of $K$. As notations of
groups, $C_n$ is the cyclic group of order $n$, $D_n$ is the
dihedral group of order $2n$, $S_n$ is the symmetric group of
degree $n$, $V_4$ is the Klein four group, i.e.\ $V_4 \simeq
C_2\times C_2$.

Recall the definition of (purely) quasi-monomial actions in
Definition \ref{d1.1} and Definition \ref{d1.9}. In particular,
whenever we say that $G$ acts on $K(x_1,\ldots,x_n)$ (or $K_\alpha
(M)$) by quasi-monomial $k$-automorphisms, it is understood that
the group $G$ and the field extension $K/k$ together with the
$G$-action on $K$ satisfy the assumptions in Definitions
\ref{d1.1} and \ref{d1.9}. The definition of (purely) monomial
actions is recalled in Example \ref{ex1.3} and Definition
\ref{d1.9}. The reader should also keep in mind the group
homomorphism $\rho_{\ul{x}}: G\to GL_n(\bm{Z})$ in Definition
\ref{d1.5}.
\end{idef}

\section{Proof of Proposition \ref{p1.6} and Proposition \ref{p1.7}}

\begin{proof}[\indent Proof of Proposition \ref{p1.6}] ~

Let $G$ act on $K(x_1,\ldots,x_n)$ by quasi-monomial $k$-automorphisms.

Without loss of generality, we may assume that $G\subset
\fn{Aut}_k(K(x_1,\ldots,x_n))$. In particular, $G$ acts faithfully
on $K(x_1,\ldots,x_n)$. Consider the group homomorphism
$\rho_{\ul{x}}: G\to GL_n(\bm{Z})$ in Definition \ref{d1.5}.

\bigskip
Step 1.
Define $N=\fn{Ker}(\rho_{\ul{x}})=\{\sigma\in G: \frac{\sigma\cdot x_j}{x_j} \in K^{\times}$ for $1\le j\le n\}$.

Define $N_0=\{\sigma\in N:\sigma(\alpha)=\alpha$ for any $\alpha\in K\}$.

If the action is a monomial action, then $k=K$ and $N_0=N$.

Clearly $N\triangleleft G$. We will show that $N_0\triangleleft G$.

If $\sigma\in N_0$, we may write $\sigma\cdot x_j=\alpha_jx_j$ for some $\alpha_j\in K^{\times}$.
For any $\tau\in G$, if $\tau\cdot x_j=c_j(\tau)\prod_{1\le i\le n} x_i^{a_{ij}}$ for some $c_j(\tau)\in K^{\times}$ and $[a_{ij}]_{1\le i,j\le n}\in GL_n(\bm{Z})$,
define $\sigma'\in N_0$ by $\sigma'\cdot \alpha=\alpha$ for all $\alpha \in K$ and $\sigma'\cdot x_j=\beta_jx_j$ with
$\beta_j=\tau^{-1}(c_j(\tau)^{-1})\cdot\tau^{-1}\sigma(c_j(\tau))\cdot \prod_{1\le i\le n} \tau^{-1}(\alpha_i)^{a_{ij}}$.
It is routine to verify $\tau^{-1}\sigma\tau=\sigma'$.
Hence $N_0\triangleleft G$ also.

\bigskip
Step 2.
We will determine $K(x_1,\ldots,x_n)^{N_0}$.

For any $\sigma\in N_0$, write $\sigma\cdot x_j=c_j(\sigma)x_j$ with $c_j(\sigma)\in K^{\times}$.
Since $N_0$ is a finite group, all the $c_j(\sigma)$ (where $\sigma\in N_0$, $1\le j\le n$) are roots of unity.
Hence they generate a cyclic group $\langle \zeta \rangle$ where $\zeta$ is some root of unity.

Choose generators $\sigma_1,\ldots,\sigma_m$ of $N_0$, i.e.\ $N_0=\langle\sigma_1,\ldots,\sigma_m\rangle$.
Define a group homomorphism
\begin{align*}
\Phi :{} & \{x_1^{\lambda_1}x_2^{\lambda_2}\cdots x_n^{\lambda_n}:\lambda_i\in \bm{Z}\} \to \langle \zeta \rangle^m \\
& X=x_1^{\lambda_1}x_2^{\lambda_2}\cdots x_n^{\lambda_n} \mapsto \left(\frac{\sigma_1(X)}{X},\frac{\sigma_2(X)}{X},\ldots,\frac{\sigma_m(X)}{X}\right).
\end{align*}

Define $M=\fn{Ker}(\Phi)$. Since $\{x_1^{\lambda_1}\cdots
x_n^{\lambda_n}:\lambda_i\in\bm{Z}\}\simeq \bm{Z}^n$, it follows
$\fn{Ker}(\Phi)\simeq \bm{Z}^n$. Thus
$M=\{y_1^{\nu_1}y_2^{\nu_2}\cdots y_n^{\nu_n}: \nu_i\in \bm{Z}\}$
where each $y_i$ is of the form $x_1^{e_1}x_2^{e_2}\cdots
x_n^{e_n}$ for some $e_j\in \bm{Z}$.

It is not difficult to see that $K(x_1,\ldots,x_n)^{N_0}=K(y_1,\ldots,y_n)$.

Moreover, $G/N_0$ acts on $K(y_1,\ldots,y_n)$ by quasi-monomial
$k$-automorphisms. For any $\tau\in G$, we will show that
$\tau\cdot y_i=\beta\cdot y_1^{\nu_1}y_2^{\nu_2}\cdots
y_n^{\nu_n}$ for some $\beta\in K^{\times}$, some $\nu_i\in
\bm{Z}$.

Since we may write $y_j=x_1^{e_1}x_2^{e_2}\cdots x_n^{e_n}$, it
follows that $\tau\cdot y_j=\beta'\cdot x_1^{e'_1} x_2^{e'_2}
\cdots x_n^{e'_n}$ for some $\beta' \in K^{\times}$, some $e'_i\in
\bm{Z}$. It remains to show that $x_1^{e'_1} x_2^{e'_2} \cdots
x_n^{e'_n} =y_1^{\mu_1} y_2^{\mu_2} \cdots y_n^{\mu_n}$ for some
$\mu_i\in \bm{Z}$, i.e.\ $x_1^{e'_1}x_2^{e'_2}\cdots x_n^{e'_n}
\in \fn{Ker} (\Phi)=\langle$ $y_1,\ldots,y_n\rangle$, which is
equivalent to showing $\sigma(x_1^{e'_1}x_2^{e'_2}\cdots
x_n^{e'_n})=x_1^{e'_1}\cdots x_n^{e'_n}$ for all $\sigma\in N_0$.

From $\tau\cdot y_j=\beta' x_1^{e'_1}\cdots x_n^{e'_n}$, if
$\sigma\in N_0$, we get $\sigma\tau
(y_j)=\sigma(\beta')\sigma(x_1^{e'_1}\cdots
x_n^{e'_n})=\beta'\sigma(x_1^{e'_1}\cdots x_n^{e'_n})$ because
$\sigma(\beta')=\beta'$. On the other hand, $\sigma\tau
(y_j)=\tau(\tau^{-1}\sigma \tau)
(y_j)=\tau(\tau^{-1}\sigma\tau(y_j))=\tau\cdot y_j$ because
$\tau^{-1}\sigma\tau\in N_0$. We get $\sigma(x_1^{e'_1}\cdots
x_n^{e'_n})=x_1^{e'_1}\cdots x_n^{e'_n}$ for any $\sigma\in N_0$.
Done.

\bigskip
Step 3. We will determine the field
$K(x_1,\ldots,x_n)^N=\{K(x_1,\ldots,x_n)^{N_0}\}^{N/N_0}$
$=K(y_1,\ldots,y_n)^{N/N_0}$. If $N_0=N$, nothing is to be proved
and we may proceed to Step 4. On the other hand, if the action of
$G$ is purely quasi-monomial, then $\sigma \cdot x_j=x_j$ for any
$\sigma \in N$ from the definition of the subgroup $N$. Thus
$K(x_1,\ldots,x_n)^N = K^N(x_1,\ldots,x_n)$. This confirms that
each $y_i$ is of the form $x_1^{e_1}x_2^{e_2}\cdots x_n^{e_n}$ in
(i) of the statement of this proposition.

For the remaining part of this step, we assume that $N_0 \neq N$.

For any $\sigma \in N$, since $\frac{\sigma\cdot x_j}{x_j} \in
K^{\times}$ for $1\le j\le n$, it follows that $\frac{\sigma\cdot
y_i}{y_i}\in K^{\times}$ for $1\le i\le n$. Write $\sigma\cdot
y_i=c_i(\sigma)y_i$ for $c_i(\sigma)\in K^{\times}$, $1\le i\le
n$.

Note that $c_i\in H^1(N/N_0,K^{\times})$, i.e.\ for any $\sigma_1,\sigma_2\in N/N_0$,
$c_i(\sigma_1\sigma_2)=\sigma_1(c_i(\sigma_2))\cdot c_i(\sigma_1)$.
Since $N/N_0$ is faithful on $K$,
it follows that $H^1(N/N_0,K^{\times})=0$ and therefore there is some element $a_i\in K^{\times}$ such that $c_i(\sigma)=\frac{a_i}{\sigma(a_i)}$ for all $\sigma\in N/N_0$.

Define $z_i=a_iy_i$.
It follows that $K(y_1,\ldots,y_n)=K(z_1,\ldots,z_n)$ and $\sigma\cdot z_i=z_i$ for all $\sigma\in N/N_0$, for all $1\le i\le n$.
Hence $K(y_1,\ldots,y_n)^{N/N_0}=K^{N/N_0}(z_1,\ldots,z_n)=K^N(z_1,\ldots,z_n)$.

In summary, if $N\ne \{1\}$, we have reduced the group from $G$ to $G/N$ and $G/N$ acts on $K^N(z_1,\ldots,z_n)$ by quasi-monomial $k$-automorphisms.

\bigskip
Step 4. Consider $\rho_{\ul{z}}: G/N \to GL_n(\bm{Z})$. If
$\fn{Ker}(\rho_{\ul{z}})$ is non-trivial, we continue the above
process and reduce the group from $G/N$ to $G/N_1$ where $N
\subsetneq N_1$. Hence the result.
\end{proof}

\begin{proof}[\indent Proof of Proposition \ref{p1.7}] ~

(1) Suppose that $G$ acts on $K(x)$ by purely quasi-monomial $k$-automorphisms.

For any $\sigma\in G$, $\sigma\cdot x=x$ or $\frac{1}{x}$.

If $\sigma\cdot x=x$ for all $\sigma\in G$, then $K(x)^G=K^G(x)=k(x)$ is $k$-rational.

From now on, we consider the case that $\sigma\cdot x=
\frac{1}{x}$ for some $\sigma\in G$.

Define $N=\{\tau\in G:\tau\cdot x=x\}$. Then $G=N\cup \sigma N$.

It follows that $K(x)^N=K^N(x)$ and $[K^N:k]\le 2$.
If $K^N=k$, then $K(x)^G=\{K(x)^N\}^{\langle\sigma\rangle}=\{K^N(x)\}^{\langle\sigma\rangle}=k(x)^{\langle\sigma\rangle}=k(x+\frac{1}{x})$ is $k$-rational.

It remains to consider the case $K^N=k(\alpha)$ is a quadratic
separable extension.

\begin{Case}{1}
$\fn{char}k\ne 2$, we may assume that $\alpha=\sqrt{a}$ for some $a\in k^{\times}$. Hence $\sigma\cdot\alpha=-\alpha$.
Define $y=\frac{1-x}{1+x}$. Then $\sigma\cdot y=-y$.
It follows that $k(\alpha)(x)^{\langle\sigma\rangle}=k(\alpha)(y)^{\langle\sigma\rangle}=k(\alpha)(\alpha y)^{\langle\sigma\rangle}=k(\alpha y)$ is $k$-rational.
\end{Case}

\begin{Case}{2}
$\fn{char}k= 2$, we may assume that $\alpha^2+\alpha=b\in k$.
Hence $\sigma\cdot\alpha=\alpha+1$. Define $y=\frac{1}{1+x}$. Then
$\sigma\cdot y=y+1$. It follows that
$k(\alpha)(x)^{\langle\sigma\rangle}=k(\alpha)(y)^{\langle\sigma\rangle}=k(\alpha)^{\langle\sigma\rangle}(y+\alpha)=k(y+\alpha)$
is $k$-rational.
\end{Case}

\bigskip
(2) Suppose that $G$ acts on $K(x)$ by quasi-monomial $k$-automorphisms.

Apply Proposition \ref{p1.6}.
Find a normal subgroup $N$ of $G$ such that (i) $K(x)^N=K^N(y)$ where $y=ax^m$ for some $a\in K^{\times}$ and some $m\in \bm{Z}$,
(ii) $\rho_{\ul{y}}: G/N\to GL_1(\bm{Z})=\{\pm 1\}$ is injective,
and (iii) $G/N$ acts on $K^N(y)$ by quasi-monomial $k$-automorphisms.

The only non-trivial case is the situation $G/N=
\langle\sigma\rangle \simeq C_2$ and $\sigma$ acts non-trivially
on $K^N$. In this case, we may write $K^N=k(\alpha)$ and
$\sigma\cdot y=\frac{b}{y}$ for some $b\in k(\alpha)^{\times}$.

Since $\sigma^2=1$ and $\sigma\cdot y=\frac{b}{y}$, it follows
that $b\in k^{\times}$.

Consider the 1-cocycle $\beta \in H^1(G/N,PGL_2(k(\alpha)))$ defined by
\[
\beta(1)=\begin{pmatrix} 1 & 0 \\ 0 & 1 \end{pmatrix}, ~
\beta(\sigma)=\begin{pmatrix} 0 & b \\ 1 & 0 \end{pmatrix} \in PGL_2(k(\alpha)).
\]

The Brauer-field $F_{1,k}(\beta)$ defined in Example \ref{ex1.4} is nothing but $k(\alpha)(y)^{\langle\sigma\rangle}$.

Note that $F_{1,k}(\beta)$ is $k$-rational if and only if the cyclic algebra $A(\beta)$ associated to $\beta$ is isomorphic to the matrix algebra $M_2(k)$ [\cite{Se}, page 160].

Let $\gamma$ be the image of $\beta$ in $H^1(G/N,PGL_2(k(\alpha)))\to H^2(G/N,k(\alpha)^{\times})$.
It is not difficult to verify that $\gamma$ is the normalized 2-cocycle with $\gamma(\sigma,\sigma)=b$.
It follows that the cyclic algebra $A(\beta)$ can be defined as
\[
A(\beta)=k(\alpha)\oplus k(\alpha)\cdot u
\]
with the relations $u^2=b$ and $u\cdot t=\sigma(t)\cdot u$ for any
$t\in k(\alpha)$.

When $\fn{char}k\ne 2$, we may choose $\alpha$ such that
$\alpha=\sqrt{a}$ for some $a\in k^{\times}$. Hence
$\sigma(\alpha)=-\alpha$ and $A(\beta)$ is just the quaternion
algebra $[a,b]_k$.

When $\fn{char}k=2$, we may choose $\alpha$ such that
$\alpha^2+\alpha=a\in k$. Hence $\sigma(\alpha)=\alpha+1$ and
$A(\beta)$ is the quaternion algebra $[a,b)_k$.

It follows that $A(\beta)\simeq M_2(k)$ if and only if the
associated norm residue $2$-symbol $(a,b)_k=0$ (if $\fn{char} k
\ne 2$), and $[a,b)_k=0$ (if $\fn{char}k=2$).

\bigskip
Finally suppose that the Brauer-field
$k(\alpha)(y)^{\langle\sigma\rangle}$ is not $k$-rational. By the
above proof, some norm residue $2$-symbol is not zero. Let $A$ be
the quaternion $k$-algebra corresponding to this norm residue
symbol. It follows that the similarity class of $A$ in $Br(k)$ is
not zero where $Br(k)$ is the Brauer group of the field $k$. Hence
$k$ is an infinite field because $Br(k')=0$ for any finite field
$k'$ by Wedderburn's Theorem [\cite{Dr}, page 73].

For the other assertion, since the Brauer-field
$k(\alpha)(y)^{\langle\sigma\rangle}$ is not $k$-rational, it is
not $k$-unirational by [\cite{Se}, page 160] (also see Example
\ref{ex1.4}).
\end{proof}

\section{Finite subgroups of \boldmath{$GL_2(\bm{Z})$}}

By [\cite{BBNWZ,GAP}], there are exactly 13 finite subgroups
contained in $GL_2(\bm{Z})$ up to conjugation. We list these
groups as follows.

\begin{theorem} \label{t3.1}
Type A. Cyclic groups
\begin{align*}
C_1 & := \{I\}, & C_2^{(1)} & := \langle -I \rangle, & C_2^{(2)} & := \langle \lambda \rangle, & C_2^{(3)} & :=\langle \tau \rangle, \\
C_3 & := \langle \sigma^2\rangle, & C_4 & := \langle\sigma\rangle, & C_6 & := \langle \rho \rangle,
\end{align*}
where
\[
I=\begin{bmatrix} 1 & 0 \\ 0 & 1 \end{bmatrix},\quad \lambda=\begin{bmatrix} 1 & 0 \\ 0 & -1 \end{bmatrix}, \quad
\tau=\begin{bmatrix} 0 & 1 \\ 1 & 0 \end{bmatrix},\quad \sigma=\begin{bmatrix} 0 & -1 \\ 1 & 0 \end{bmatrix}, \quad
\rho=\begin{bmatrix} 1 & -1 \\ 1 & 0 \end{bmatrix}.
\]

Type B. Non-cyclic groups :
\begin{align*}
V_4^{(1)} &:= \langle\lambda,-I\rangle, & V_4^{(2)} &:= \langle\tau,-I\rangle, & S_3^{(1)} &:= \langle \rho^2,\tau\rangle, & S_3^{(2)} &:= \langle \rho^2,-\tau\rangle, \\
D_4 &:= \langle \sigma,\tau \rangle, & D_6 &:= \langle\rho,\tau\rangle.
\end{align*}
\end{theorem}

Note that $\lambda^2=\tau^2=I$, $\sigma^2=\rho^3=-I$, and $\tau\sigma=\lambda$.

\begin{lemma}[{[\cite{HK2}, Lemma 2.7]}] \label{l3.1}
Let $k$ be a field and $-I\in GL_2(\bm{Z})$ act on $k(x,y)$ by a
$k$-automorphism defined as
\[
-I: x\mapsto \tfrac{a}{x},~ y\mapsto \tfrac{b}{y}, ~ a,b\in k^{\times}.
\]
Then $k(x,y)^{\langle -I\rangle}=k(u,v)$ where
\[
u=\frac{x-\frac{a}{x}}{xy-\frac{ab}{xy}}, \quad v=\frac{y-\frac{b}{y}}{xy-\frac{ab}{xy}}.
\]
\end{lemma}

\begin{lemma} \label{l3.2}
Let $k$ be a field and $-I\in GL_2(\bm{Z})$ act on $k(x,y)$ by a
$k$-automorphism defined as
\[
-I: x\mapsto \tfrac{1}{x},~ y\mapsto \tfrac{1}{y}.
\]
Then $k(x,y)^{\langle -I\rangle}=k(s,t)$ where
\begin{equation}
s=\frac{xy+1}{x+y},\quad
t=\begin{cases}
\frac{xy-1}{x-y} & \text{if }\fn{char}k\ne 2, \\[4pt]
\frac{x(y^2+1)}{y(x^2+1)} & \text{if } \fn{char} k=2.
\end{cases} \label{eq1}
\end{equation}
\end{lemma}

\begin{proof}
When $\fn{char}k\ne 2$, see [\cite{HHR}, page 1176; \cite{HKY}, Lemma 3.4].
When $\fn{char}k=2$, $s$, $t$, $x$ and $y$ satisfy the relations
\[
y^2+\frac{(s^2+1)(t+1)}{s}y+1=0, \quad x=\frac{sy+1}{s+y}.
\]
Hence the result.
\end{proof}

\begin{lemma} \label{l3.3}
Let $k$ be a field and $D_6=\langle \rho,\tau\rangle$ act on
$k(x,y)$ by purely monomial $k$-automorphisms defined as
\[
\rho: x\mapsto xy,~ y\mapsto \tfrac{1}{x}, \quad \tau: x\mapsto y,~ y\mapsto x.
\]
Then $k(x,y)^{\langle\rho^2\rangle}=k(S,T)$ where
\begin{equation}
\begin{aligned}
S &= \frac{x^2y+xy^2-3xy+1}{x^2y^2-3xy+x+y}, \\
T &= \begin{cases}
\frac{(xy+y+1)(x^2y^2-x^2y+x^2-xy-x+1)}{(xy+x+1)(x^2y^2-3xy+x+y)} & \text{if } \fn{char} k\ne 3, \\[4pt]
\frac{x(x^3y^3+y^3+1)}{y(x^3y^3+x^3+1)} & \text{if } \fn{char}
k=3.
\end{cases}
\end{aligned} \label{eq2}
\end{equation}
Moreover, $D_6/\langle\rho^2\rangle=\langle \rho,\tau\rangle$ acts on $k(S,T)$ by
\[
\rho:S\mapsto \tfrac{1}{S},~T\mapsto \begin{cases} \frac{S+\frac{1}{S}-1}{T}, & \\ \frac{1}{T}, & \end{cases} \quad
\tau: S\mapsto S,~ T\mapsto \begin{cases} \frac{S(S+\frac{1}{S}-1)}{T} & \text{if }\fn{char}k\ne 3, \\
\frac{1}{T} & \text{if } \fn{char} k=3. \end{cases}
\]
\end{lemma}

\begin{proof}
\begin{Case}{1} $\fn{char}k\ne 2,3$. \end{Case}

Define $z:= 1/xy$. The $k$-automorphism $\rho^2$ acts on
$k(x,y)=k(x,y,z)$ by $\rho^2:x\mapsto y\mapsto z\mapsto x$ and the
fixed field under the action of $\rho^2$ is given by
\begin{equation}
k(x,y,z)^{\langle\rho^2\rangle} = k(A,B,C,D)=k(A,B,D) \label{eq3}
\end{equation}
where
\begin{gather*}
A=x+y+z=\frac{x^2y+xy^2+1}{xy}, \quad B=xy+yz+zx=\frac{x^2y^2+x+y}{xy}, \\
C=xyz=1, \quad
D=(x-y)(y-z)(x-z)=\frac{(x-y)(xy^2-1)(x^2y-1)}{x^2y^2}.
\end{gather*}
Define $A_0=A/3$, $B_0=B/3$, $D_0=D/9$. Then the generators $A_0$,
$B_0$ and $D_0$ of $k(x,y,z)^{\langle
\rho^2\rangle}=k(A_0,B_0,D_0)$ satisfy the relation
\begin{equation}
3D_0^2=-4A_0^3+3A_0^2B_0^2+6A_0B_0-4B_0^3-1. \label{eq4}
\end{equation}
Find the singularities of \eqref{eq4}. We get $A_0-1=B_0-1=D_0=0$.
Put $A_1=A_0-1$, $B_1=B_0-1$. Then the equation \eqref{eq4}
becomes
\[
3D_0^2=-4A_1^3+3A_1^2B_1^2+6A_1^2B_1-9A_1^2+6A_1B_1^2+18A_1B_1-4B_1^3-9B_1^2
\]
with a singular point at the origin $P=(D_0=A_1=B_1=0)$.
Blowing up at the point $P$ by defining $A_2=A_1/B_1$, $D_1=D_0/B_1$, we have
\begin{equation}
3D_1^2=-4A_2^3B_1+3A_2^2B_1^2+6A_2^2B_1-9A_2^2+6A_2B_1+18A_2-4B_1-9. \label{eq5}
\end{equation}
This equation \eqref{eq5} also has singularities $A_2-B_1-2=B_1^2+3B_1+3=D_1=0$.
Blowing up by defining $A_3=\frac{A_2-B_1-2}{B_1^2+3B_1+3}$, $D_2=\frac{D_1}{B_1^2+3B_1+3}$, we get
\begin{equation}
3D_2^2=-4A_3^3B_1^3-12A_3^3B_1^2-12A_3^3B_1-9A_3^2B_1^2-18A_3^2B_1-9A_3^2-6A_3B_1-6A_3-1. \label{eq6}
\end{equation}
Define $B_2=B_1A_3$. Then we get $k(A_3,B_1,D_2)=k(A_3,B_2,D_2)$
with the relation
\begin{equation}
3D_2^2=-12A_3^2B_2-9A_3^2-12A_3B_2^2-18A_3B_2-6A_3-4B_2^3-9B_2^2-6B_2-1. \label{eq7}
\end{equation}
The equation \eqref{eq7} still has a singular point $P=(A_3=B_2+1=D_2=0)$.
Blowing up again by defining $A_4=\frac{A_3}{B_2+1}$, $B_3=B_2+1$, $D_3=\frac{D_2}{B_2+1}$,
we have $k(A_3,B_2,D_2)=k(A_4,B_3,D_3)$ and
\begin{equation}
3D_3^2=-12A_4^2B_3+3A_4^2-12A_4B_3+6A_4-4B_3+3. \label{eq8}
\end{equation}
Hence we obtain $k(x,y,z)^{\langle\rho^2\rangle}=k(A_4,B_3,D_3)=k(A_4,D_3)$ because the equation \eqref{eq8} is linear in $B_3$.

\bigskip
The action of $D_6=\langle\rho,\tau\rangle$ on $k(A_4,D_3)$ is
given by
\begin{align*}
\rho &: A_4 \mapsto \tfrac{A_4^2+2A_4+3D_3^2+1}{3A_4^2+2A_4-3D_3^2-1}, ~ D_3\mapsto -\tfrac{4A_4D_3}{3A_4^2+2A_4-3D_3^2-1}, \\
\tau &: A_4\mapsto A_4, ~ D_3\mapsto -D_3.
\end{align*}
Define $A_5=A_4+\frac{1}{3}$. Then the actions of $\rho$ and
$\tau$ on $k(A_5,D_3)$ are
\begin{align*}
\rho &: A_5\mapsto \tfrac{2(3A_5^2+2A_5+3D_3^2)}{9(A_5^2-D_3^2)-4}, ~ D_3\mapsto -\tfrac{4(3A_5-1)D_3}{9(A_5^2-D_3^2)-4}, \\
\tau &: A_5\mapsto A_5,~ D_3\mapsto -D_3.
\end{align*}
Put $A_6=3(A_5+D_3)$, $D_4=3(A_5-D_3)$. Then we have $k(A_5,D_3)=k(A_6,D_4)$ and
\[
\rho: A_6\mapsto \tfrac{2(2A_6+D_4^2)}{A_6D_4-4},~ D_4\mapsto\tfrac{2(A_6^2+2D_4)}{A_6D_4-4},\quad
\tau: A_6 \mapsto D_4,~ D_4\mapsto A_6.
\]
We put
\begin{equation}
S=\frac{A_6D_4-4}{2(A_6+D_4-2)}, \quad T=-\frac{D_4^2-2D_4+4}{2(A_6+D_4-2)}. \label{eq9}
\end{equation}
Then we get $k(A_6,D_4)=k(S,T)$ and
\[
\rho: S\mapsto \tfrac{1}{S},~ \tau \mapsto\tfrac{S+\frac{1}{S}-1}{T}, \quad \tau: S\mapsto S,~ T\mapsto \tfrac{S(S+\frac{1}{S}-1)}{T}.
\]
With the aid of computers, it is not difficult to express $S$ and
$T$ in terms of $x$ and $y$ as in the statement of this lemma.

\bigskip
\begin{Case}{2} $\fn{char}k=2$. \end{Case}

The transcendental basis $S$, $T$ obtained in \eqref{eq9} of Case
1 is also valid in case $\fn{char}k=2$. In fact, the formulae of
$S$ and $T$ in \eqref{eq2}, when $\fn{char}k=2$, become
\[
S=\frac{x^2y+xy^2+xy+1}{x^2y^2+xy+x+y},\quad T=\frac{(xy+y+1)(x^2y^2+x^2y+x^2+xy+x+1)}{(xy+x+1)(x^2y^2+xy+x+y)}
\]

Both of them are fixed by $\rho^2$. Moreover, they satisfy the
following relations
\begin{gather*}
y^3+\frac{S^2T+S^2+S+T^2+1}{S^3+ST^2+ST+T^2+1}y^2+\frac{S^3+S^2+ST^2+S+T}{S^3+ST^2+ST+T^2+1}y+1=0, \\
x=\frac{(T+1)y+S+T}{(S+T+1)y^2+(S+1)y+S+T+1}.
\end{gather*}
From $k(S,T)\subset k(x,y)^{\langle\rho^2\rangle} \subset k(x,y)$
and $[k(x,y):k(S,T)]\le 3$, we find that
$k(S,T)=k(x,y)^{\langle\rho^2\rangle}$.

\bigskip
\begin{Case}{3} $\fn{char}k=3$. \end{Case}

The basis $S$, $T$ in \eqref{eq9} is well defined also for the
case of $\fn{char}k=3$, but they collapse because $S+2T+1=0$. We
will find another transcendental basis $S$, $T$.

By \eqref{eq3}, we obtain $k(x,y)^{\langle\rho^2\rangle}=k(A,B,D)$ and $D^2=-A^3+A^2B^2-B^3$.
Put $A_2=\frac{A-B}{A+B}$, $D_2=\frac{D}{AB}$.
Then we get $k(A,B,D)=k(A_2,B,D_2)$ and
\[
B(A_2+1)(A_2^2D_2^2-A_2^2-D_2^2+1)+1=0.
\]
Hence $k(x,y)^{\langle\rho^2\rangle}=k(A_2,D_2)$. The actions of
$\rho$ and $\tau$ on $k(A_2,D_2)$ are given by
\[
\rho: A_2\mapsto -A_2,~ D_2\mapsto -D_2, \quad \tau: A_2\mapsto A_2,~ D_2\mapsto -D_2.
\]
Define $S=\frac{1+A_2}{1-A_2}$, $T=\frac{1+D_2}{1-D_2}$. Then we
get $k(A_2,D_2)=k(S,T)$ and
\[
\rho: S\mapsto \tfrac{1}{S},~ T\mapsto \tfrac{1}{T}, \quad \tau: S\mapsto S,~ T\mapsto \tfrac{1}{T}.
\]
We may also express $S$ and $T$ in terms of $x$ and $y$ as in the
statement of this lemma. Done.
\end{proof}

\section{Proof of Theorem \ref{t1.8}}

Throughout this section, we adopt the following convention : If
$K/k$ is a quadratic separable extension, we will write
$K=k(\alpha)$ with $\ol{\alpha}=-\alpha$ (resp.\ $\alpha+1$) if
$\fn{char}k\ne 2$ (resp.\ $\fn{char}k=2$) where $\ol{\alpha}$ is
the image of $\alpha$ under the unique non-trivial
$k$-automorphism of $K$. In order to shorten the wording, we
simply say that $K/k$ is a quadratic separable extension and
$K=k(\alpha)$ with $\alpha$ suitably chosen.

\begin{lemma} \label{l4.1}
Let $k$ be a field, $f(x)\in k[x]$, and $K/k$ a separable
quadratic field extension with $K=k(\alpha)$ where $\alpha^2=a \in
k$ \rm{(}if $\fn{char}k\ne 2$\rm{)} and $\alpha^2+\alpha=a \in k$
\rm{(}if $\fn{char}k=2$\rm{)}. Write $Gal(K/k)= \langle \sigma
\rangle$. Extend the action of $\sigma$ to $K(x,y)$ by
\[
\sigma: \alpha \mapsto \ol{\alpha},~ x \mapsto x,~ y\mapsto
\tfrac{f(x)}{y}.
\]
Denote by $\fn{Br}(K/k)$ the subgroup consisting of those elements
in the Brauer group $\fn{Br}(k)$, which are split over $K$; denote
by $(a,b)_k$ the norm residue $2$-symbol over $k$, and by
$[a,b)_k$ the norm residue $2$-symbol over $k$ where the first
variable is additive and the second variable is multiplicative.
\begin{enumerate}
\item[{\rm (1)}] When $f(x)=b$, $K(x,y)^{\langle\sigma\rangle}$ is
rational over $k$ if and only if {\rm (i)} $(a,b)_k=0$ when
$\fn{char}k\ne 2$, or {\rm (ii)} $[a,b)_k=0$ when $\fn{char}k=2$.
\item[{\rm (2)}] When $\deg f(x)=1$,
$K(x,y)^{\langle\sigma\rangle}$ is always rational over $k$.
\item[{\rm (3)}] When $\fn{char}k\ne 2$ and $f(x)=b(x^2-c)$ for
some $b,c\in k\backslash \{0\}$, then
$K(x,y)^{\langle\sigma\rangle}$ is rational over $k$ if and only
if $(a,b)_k\in \fn{Br}(k(\sqrt{ac})/k)$. \item[{\rm (4)}] When
$\fn{char}k=2$ and $f(x)=b(x^2+x+c)$ for some $b,c\in k$ with
$b\ne 0$, then $K(x,y)^{\langle\sigma\rangle}$ is rational over
$k$ if and only if $[a,b)_k\in \fn{Br}(k(\beta)/k)$ where
$\beta^2+\beta=a+c$. \item[{\rm (5)}] When $\fn{char}k=2$ and
$f(x)=b(x^2+c)$ for some $b,c\in k\backslash \{0\}$, then
$K(x,y)^{\langle\sigma\rangle}$ is rational over $k$ if and only
if $[a,b)_k\in \fn{Br}(k(\sqrt{c})/k)$.
\end{enumerate}
Moreover, if $K(x,y)^{\langle\sigma\rangle}$ is not $k$-rational,
then $k$ is an infinite field, the Brauer group $\fn{Br}(k)$ is
non-trivial, and $K(x,y)^{\langle\sigma\rangle}$ is not
$k$-unirational.
\end{lemma}

\begin{proof}
Except for the last statement, the theorem was proved in
[\cite{HKO}, Theorem 6.7]; also see [\cite{HK2}, Theorem 2.4;
\cite{Ka2}, Theorem 4.2].

Now we will prove the last statement. The field $k$ is infinite
because some norm residue $2$-symbol is not zero and the proof to
proving $k$ is infinite in the proof of Proposition \ref{p1.7} is
valid at the present situation. It remains to show that
$K(x,y)^{\langle\sigma\rangle}$ is not $k$-unirational.

As an illustration, consider (4) where $\fn{char}k=2$ and
$f(x)=b(x^2+x+c)$. It is easy to show that
$K(x,y)^{\langle\sigma\rangle}$ is $k$-isomorphic to the function
field of the hypersurface $\{P(X,Y,U)= X^2-XY-aY^2-b(U^2-U-c)=0
\}$ in the $3$-dimensional affine space over $k$ where $X,Y,U$ are
the coordinates (see, for example, [\cite{HKO}, Theorem, page
402]). By [\cite{HKO}, Theorem 2.2], this function field is
$k$-rational if and only if it is $k$-unirational (note that
$k$-unirational is called subrational over $k$ in [\cite{HKO},
page 386]). This finishes the proof.

Similarly, the field $K(x,y)^{\langle\sigma\rangle}$ in (5)
corresponds to the hypersurface defined by $\{P(X,Y,U)=
X^2-XY-aY^2-b(U^2-c)=0 \}$; the field in (3) corresponds to the
hypersurface defined by $\{P(X,Y,U)= X^2-aY^2-b(U^2-c)=0 \}$; the
field in (1) corresponds to the hypersurface defined by
$\{P(X,Y,U)= X^2-aY^2-b=0 \}$ (resp. $\{P(X,Y,U)= X^2-XY-aY^2-b=0
\}$. Apply [\cite{HKO}, Theorem 2.2] to these fields. Hence the
result.
\end{proof}

\begin{example} \label{ex4.2}
Take $k=\bm{Q}$, $a=-1$, $f(x)=-(x^2+1)$. By Lemma \ref{l4.1}, we
find that $\bm{Q}(\sqrt{-1})(x,y)^{\langle\sigma\rangle}$ is not
$\bm{Q}$-unirational. By defining $2u=y+\tfrac{f(x)}{y}$ and
$2v=\sqrt{-1} \, (y-\tfrac{f(x)}{y})$, we obtain
$\bm{Q}(\sqrt{-1})(x,y)^{\langle\sigma\rangle}=\bm{Q}(x,u,v)$ with
the relation $u^2+v^2=-x^2-1$. Note that the field $\bm{Q}(x,u,v)$
has no $\bm{Q}$-rational place (i.e. $\bm{Q}$-rational point).

On the other hand, take $k=\bm{R}$, $a=-1$, $f(x)=x^3 - 3x$. Then
$\bm{R}(\sqrt{-1})(x,y)^{\langle\sigma\rangle}$ $\simeq
\bm{R}(x,u,v)$ with the relation $u^2+v^2=x^3 - 3x$. We claim that
$\bm{R}(x,u,v)$ is $\bm{R}$-unirational, but not
$\bm{R}$-rational. For the $\bm{R}$-unirationality, the map $\phi
: \bm{R}[U,V,X]/\langle U^2 + V^2 -X^3 +3X \rangle \to
\bm{R}(s,t)$ defined by $\phi(U)=\frac{s(3+s^2)(1-t^2)-2
\sqrt{2}t}{1+t^2}$, $\phi(V)=\frac{2st(3+s^2)+
\sqrt{2}(1-t^2)}{1+t^2}$, $\phi(X)=2+s^2$ gives an embedding of
$\bm{R}(x,u,v)$ into $\bm{R}(s,t)$ [\cite{Oj}, page 9]. The
irrationality of $\bm{R}(x,u,v)$ follows from Iskovskikh's
criterion (see, for example, [\cite{Ka2}, Theorem 4.3]); the
reader may consult [\cite{Oj}, pages 10-12; pages 81-83] for other
proofs. The example of a real algebraic surface which is
unirational, but not rational is constructed by B. Segre in 1951
[\cite{Oj}, page iv].

Nagata asks the following question [\cite{Na}, page 90] : Let $k$
be an algebraically closed field, $f_1,f_2,f_3$ be non-zero
polynomials in the polynomial ring $k[x_1, \ldots, x_n]$, $L$ be
the field $k(x_1, \ldots, x_n,u,v)$ satisfying the relation
$f_1u^2+f_2v^2=f_3$. Is it possible to find such a field $L$ which
is not $k$-unirational? So far as we know, this question is still
an open problem.
\end{example}

\bigskip
\begin{proof}[\indent Proof of Theorem \ref{t1.8}] ~

Let $G$ act on $K(x,y)$ by purely quasi-monomial $k$-automorphisms
and $N$, $H$ the subgroups of $G$ defined in the statement of
Theorem \ref{t1.8}.

Note that $K(x,y)^G=\{K(x,y)^N\}^{G/N}$ and $G/N$ acts on $K^N(x,y)$ by purely monomial $k$-automorphisms.
Without loss of generality, we may assume $N=\{1\}$.
Thus $\rho:G\to GL_2(\bm{Z})$ is injective where $\rho$ is the group homomorphism defined in Definition \ref{d1.5}.
In particular, $G$ is isomorphic to a finite subgroup of $GL_2(\bm{Z})$ and we may apply Theorem \ref{t3.1} to write down the action of $G$ on $x$ and $y$.

If $H=\{1\}$, i.e.\ $G$ acts faithfully on $K$, then $K(x,y)^G$ is
$k$-rational by Theorem \ref{t1.11}.

If $H=G$, i.e.\ $k=K$, then $G$ acts on $k(x,y)$ by purely
monomial $k$-automorphisms. Thus $k(x,y)^G$ is $k$-rational by
Theorem \ref{t1.13}.

It remains to consider the case $\{1\}\subsetneq H\subsetneq G$ and $H\triangleleft G$.

Since $\rho: G\to GL_2(\bm{Z})$ is injective, $\rho(H)$ is a
non-trivial proper normal subgroup of $\rho(G)$. From Theorem
\ref{t3.1}, it is easy to see that only 8 groups have a
non-trivial proper normal subgroup: $V_4^{(1)}$, $V_4^{(2)}$,
$C_4$, $S_3^{(1)}$, $S_3^{(2)}$, $C_6$, $D_4$ and $D_6$.

\bigskip
\begin{Case}{1} $G\simeq V_4^{(1)}=\langle \lambda,-I\rangle$. \end{Case}

The non-trivial proper normal subgroups $H$ of $G$ are $\langle -I \rangle$, $\langle \lambda \rangle$, and $\langle -\lambda \rangle$.

\medskip
Subcase 1.1. $(G,H)\simeq (V_4^{(1)},\langle -I\rangle)$. Thus
$\lambda$ acts faithfully on $K$ and $[K:k]=2$. Write
$K=k(\alpha)$ with $\lambda(\alpha)=\ol{\alpha}$ where $\alpha$ is
chosen suitably. The group $G=\langle -I,\lambda\rangle$ acts on
$K(x,y)$ by
\[
-I: \alpha\mapsto \alpha,~ x\mapsto \tfrac{1}{x},~ y\mapsto \tfrac{1}{y}, \quad \lambda:\alpha\mapsto \ol{\alpha},~ x\mapsto x,~ y\mapsto\tfrac{1}{y}.
\]
By Lemma \ref{l3.2} we have $K(x,y)^H=K(s,t)$ where $s$ and $t$
are given as in \eqref{eq1}, and $\lambda$ acts on $K(s,t)$ by
\[
\lambda: \alpha\mapsto \ol{\alpha},~ s\mapsto \tfrac{1}{s},~ t\mapsto \begin{cases}
\frac{1}{t} & \text{if }\fn{char}k\ne 2, \\ t & \text{if } \fn{char}k=2. \end{cases}
\]
Hence
\[
K(x,y)^G=K(s,t)^{\langle\lambda\rangle}=\begin{cases}
k(\alpha\frac{s+1}{s-1},\alpha\frac{t+1}{t-1}) & \text{if }\fn{char}k\ne 2, \\
k(\alpha+\frac{1}{s+1},t) & \text{if }\fn{char}k=2, \end{cases}
\]
is $k$-rational.

\medskip
Subcase 1.2. $(G,H)\simeq (V_4^{(1)},\langle\lambda\rangle)$.
Again $[K:k]=2$ and write $K=k(\alpha)$ with $\alpha$ suitably chosen.
The group $G$ acts on $K(x,y)$ by
\[
\lambda:\alpha\mapsto \alpha,~ x\mapsto x,~ y\mapsto\tfrac{1}{y},
\quad -I:\alpha\mapsto \ol{\alpha},~ x\mapsto\tfrac{1}{x},~
y\mapsto \tfrac{1}{y}.
\]

We have $K(x,y)^H=K(u,v)$ where $u=x$ and $v=y+\frac{1}{y}$, and
$-I$ acts on $K(u,v)$ by
\[
-I:\alpha\mapsto\ol{\alpha},~ u\mapsto\tfrac{1}{u},~ v\mapsto v.
\]
Hence
\[
K(x,y)^G=K(u,v)^{\langle -I\rangle}=\begin{cases}
k(\alpha\frac{u+1}{u-1},v) & \text{if }\fn{char}k\ne 2, \\
k(\alpha+\frac{1}{u+1},v) & \text{if }\fn{char}k=2 \end{cases}
\]
is $k$-rational.

\medskip
Subcase 1.3. $(G,H)\simeq (V_4^{(1)},\langle -\lambda \rangle)$.

The group $G$ acts on $K(x,y)$ by
\[
-\lambda : \alpha\mapsto\alpha,~ x\mapsto \frac{1}{x},~ y\mapsto
y, \quad -I : \alpha\mapsto \bar{\alpha},~ x\mapsto \tfrac{1}{x},~
y\mapsto \tfrac{1}{y}
\]
where $[K:k]=2$ and $K=k(\alpha)$ with $\alpha$ chosen suitably.
This is essentially the same situation as in Subcase 1.2.

\bigskip
\begin{Case}{2} $G\simeq V_4^{(2)}=\langle \tau,-I \rangle$. \end{Case}

The non-trivial proper normal subgroups $H$ of $G$ are $\langle -I\rangle$, $\langle \tau \rangle$ and $\langle -\tau\rangle$.

\medskip
Subcase 2.1. $(G,H)\simeq (V_4^{(2)}, \langle -I\rangle)$. We have
$[K:k]=2$ and $K=k(\alpha)$ with $\alpha$ chosen suitably.

The proof is similar by using Lemma \ref{l3.1}. The details are
omitted.

\medskip
Subcase 2.2. $(G,H)\simeq (V_4^{(2)},\langle \tau \rangle)$. Write
$K=k(\alpha)$ with $\alpha$ suitably chosen. The group $G$ acts on
$K(x,y)$ by
\[
\tau:\alpha\mapsto \alpha,~ x\mapsto y,~ y\mapsto x, \quad
-I:\alpha\mapsto \ol{\alpha},~ x\mapsto\tfrac{1}{x},~ y\mapsto \tfrac{1}{y}.
\]
We have $K(x,y)^{\langle\tau\rangle}=K(u,v)$ where $u:=x+y$ and $v:=(x+y)/xy$, and $-I$ acts on $K(u,v)$ by
\[
-I:\alpha\mapsto\ol{\alpha},~ u\mapsto v,~ v\mapsto u.
\]
Hence
\[
K(x,y)^G=K(s,t)^{\langle -I\rangle}=\begin{cases}
k(u+v,\alpha(u-v)) & \text{if }\fn{char}k\ne 2, \\
k(u+v,\alpha+\frac{u}{u+v}) & \text{if }\fn{char}k=2 \end{cases}
\]
is $k$-rational.

\medskip
Subcase 2.3. $(G,H)=(V_4^{(2)},\langle -\tau \rangle)$.
The group $G$ acts on $K(x,y)$ by
\[
-\tau: \alpha\mapsto\alpha,~ x\mapsto\tfrac{1}{y},~ y\mapsto \tfrac{1}{x}, \quad
-I:\alpha\mapsto\ol{\alpha},~x\mapsto\tfrac{1}{x},~ y\mapsto \tfrac{1}{y}
\]
where $[K:k]=2$ and $K=k(\alpha)$ with $\alpha$ chosen suitably.
Define $y'=1/y$. We have
\[
-\tau:\alpha\mapsto \alpha,~ x\mapsto y',~ y'\mapsto x, \quad
-I:\alpha\mapsto\ol{\alpha},~x\mapsto \tfrac{1}{x},~ y'\mapsto \tfrac{1}{y'}
\]

This is the same action as in Subcase 2.2. Done.

\bigskip
\begin{Case}{3} $G\simeq C_4=\langle\sigma\rangle$. \end{Case}

The non-trivial proper normal subgroup $N$ of $G$ is $\langle -I
\rangle$. The action is given by
\[
-I=\sigma^2:\alpha\mapsto \alpha,~ x\mapsto\tfrac{1}{x},~y\mapsto \tfrac{1}{y},\quad
\sigma:\alpha\mapsto \ol{\alpha},~ x\mapsto y,~ y\mapsto\tfrac{1}{x}
\]
where $[K:k]=2$ and $K=k(\alpha)$ with $\alpha$ chosen suitably.
By Lemma \ref{l3.2} we have $K(x,y)^N=K(s,t)$ where $s$ and $t$
are given as in \eqref{eq1}, and $\sigma$ acts on $K(s,t)$ by
\[
\sigma:\alpha\mapsto \ol{\alpha},~ s\mapsto\tfrac{1}{s},~ t\mapsto \begin{cases}
-\frac{1}{t} & \text{if } \fn{char}k\ne 2, \\
\frac{1}{t} & \text{if } \fn{char}k=2. \end{cases}
\]

If $\fn{char}k=2$ then $K(x,y)^G=K(s,t)^{\langle \sigma\rangle}$
is $k$-rational by Theorem \ref{t1.11}.

When $\fn{char}k\ne 2$, apply Lemma \ref{l4.1}. By defining
$a=\alpha^2\in k$ and $x=\alpha\frac{s+1}{s-1}$, we find that
$K(x,y)^G$ is $k$-rational if and only if $(a,-1)_k=0$. If
$K(x,y)^G$ is not $k$-rational, it is not $k$-unirational by the
last assertion of Lemma \ref{l4.1}. This finishes the proof of the
first exceptional case in Theorem \ref{t1.8}.

\bigskip
\begin{Case}{4} $G\simeq S_3^{(1)}=\langle \rho^2,\tau \rangle$. \end{Case}

The non-trivial proper normal subgroup $N$ of $G$ is $\langle \rho^2\rangle$.
We consider the following actions:
\[
\rho^2:\alpha\mapsto\alpha,~ x\mapsto y,~ y\mapsto \tfrac{1}{xy}, \quad
\tau:\alpha \mapsto \ol{\alpha},~ x\mapsto y,~ y\mapsto x
\]
where $[K:k]=2$ and $K=k(\alpha)$ with $\alpha$ chosen suitably.
We have $K(x,y)^{\langle \rho^2 \rangle}=K(S,T)$ where $S$ and $T$
are defined in Lemma \ref{l3.3}, and the action of $\tau$ on
$K(S,T)$ is given by
\[
\tau:\alpha\mapsto\ol{\alpha},~ S\mapsto S,~ T\mapsto \begin{cases}
\frac{S(S+\frac{1}{S}-1)}{T} & \text{if } \fn{char}k\ne 3, \\
\frac{1}{T} & \text{if } \fn{char}k=3. \end{cases}
\]
Hence $K(x,y)^G=K(S,T)^{\langle \tau\rangle}$ is $k$-rational by
Lemma \ref{l4.1}.

\bigskip
\begin{Case}{5} $G\simeq S_3^{(2)}=\langle \rho^2,-\tau\rangle$. \end{Case}

The non-trivial proper normal subgroup $N$ of $G$ is $\langle \rho^2 \rangle$.
We consider the following actions:
\[
\rho^2:\alpha\mapsto\alpha,~ x\mapsto y,~ y\mapsto \tfrac{1}{xy}, \quad
-\tau: \alpha\mapsto \ol{\alpha},~ x\mapsto \tfrac{1}{y},~ y\mapsto\tfrac{1}{x}
\]
where $K=k(\alpha)$ and $[K:k]=2$ with $\alpha$ suitably chosen.
We have $K(x,y)^{\langle\rho^2\rangle}=K(S,T)$ where $S$ and $T$
are defined in Lemma \ref{l3.3}, and the action of $-\tau$ on
$K(S,T)$ is given by
\[
-\tau: \alpha\mapsto\ol{\alpha},~S\mapsto \tfrac{1}{S},~ T\mapsto \begin{cases}
\frac{T}{S} & \text{if } \fn{char}k\ne 3, \\
T & \text{if }\fn{char}k=3. \end{cases}
\]

$K(x,y)^G=K(S,T)^{\langle -\tau \rangle}$ is $k$-rational by
Theorem \ref{t1.11}.

\bigskip
\begin{Case}{6} $G\simeq C_6=\langle\rho\rangle$. \end{Case}

The non-trivial proper normal subgroups $N$ of $G$ are $\langle -I
\rangle$ and $\langle \rho^2 \rangle$.

\medskip
Subcase 6.1. $(G,H)\simeq (C_6,\langle -I\rangle)$.
The group $G$ acts on $K(x,y)$ by
\[
-I:\alpha\mapsto \alpha,~ x\mapsto\tfrac{1}{x},~ y\mapsto \tfrac{1}{y}, \quad
\rho:\alpha\mapsto \ol{\alpha},~ x\mapsto xy,~ y\mapsto \tfrac{1}{x}
\]
where $K=k(\alpha)$ and $[K:k]=3$ with $\alpha$ as before. By
Lemma \ref{l3.2}, we get $K(x,y)^{\langle -I\rangle}=K(s,t)$ where
$s$ and $t$ are given as in \eqref{eq1}, and $\rho$ acts on
$K(s,t)$ by
\[
\rho: s\mapsto\left\{\begin{array}{@{}l@{}} \frac{s-t}{s+t-2st}, \\ \frac{t}{s(t+1)+1}, \end{array}\right. ~
t\mapsto \begin{cases} \frac{-s+t}{s+t+2st} & \text{if }\fn{char}k\ne 2, \\
\frac{1}{s(t+1)} & \text{if }\fn{char}k=2. \end{cases}
\]
Define
\begin{alignat*}{3}
A &= \frac{-2t}{s-t+2st}, &\q B &= \frac{s+t-2st}{2t}, &\q & \text{if }\fn{char}k\ne 2; \\
A &= s(t+1), & B &= \frac{t}{s(t+1)}, && \text{if }\fn{char}k=2.
\end{alignat*}
Then $K(x,y)^{\langle -I\rangle}=K(s,t)=K(A,B)$ and $\rho$ acts on $K(A,B)$ by
\[
\rho:\alpha\mapsto\ol{\alpha},~ A\mapsto B,~ B\mapsto \tfrac{1}{AB}.
\]

$K(x,y)^G=K(A,B)^{\langle\rho\rangle}$ is $k$-rational by Theorem
\ref{t1.11}.

\medskip
Subcase 6.2. $(G,H)\simeq (C_6,\langle \rho^2\rangle)$.
The group $G$ acts on $K(x,y)$ by
\[
\rho^2:\alpha\mapsto \alpha,~ x\mapsto y,~ y\mapsto \tfrac{1}{xy}, \quad
\rho:\alpha\mapsto \ol{\alpha},~ x\mapsto xy,~ y\mapsto \tfrac{1}{x}
\]
where $K=k(\alpha)$ and $[K:k]=2$ with $\alpha$ as before. We have
$K(x,y)^{\langle \rho^2 \rangle}=K(S,T)$ where $S$ and $T$ are
defined in Lemma \ref{l3.3}, and $\rho$ acts on $K(S,T)$ by
\[
\rho:\alpha\mapsto\ol{\alpha},~ S\mapsto \tfrac{1}{S},~ T\mapsto \begin{cases}
(S+\frac{1}{S}-1)/T & \text{if } \fn{char} k\ne 3, \\
\frac{1}{T} & \text{if }\fn{char}k=3. \end{cases}
\]

When $\fn{char}k=3$,
$K(x,y)^G=k(\alpha\frac{S+1}{S-1},\alpha\frac{T+1}{T-1})$ is
rational over $k$.

When $\fn{char}k=2$, define
\[
U:=\frac{S}{S+1}+\alpha,\quad V:=\frac{T}{S+1}.
\]

Then $K(S,T)=K(U,V)$ and $\rho$ acts on $K(U,V)$ by
\[
\rho: \alpha\mapsto \alpha+1,~ U\mapsto U,~ V\mapsto \tfrac{U^2+U+a+1}{V}
\]
where $a=\alpha(\alpha+1)$. By Lemma \ref{l4.1} (4),
$K(x,y)^G=K(U,V)^{\langle\rho\rangle}$ is $k$-rational because
$[a,1)_k=0$.

When $\fn{char}k\ne 2,3$, define
\[
U:=\frac{S+1}{S-1}, \quad V:=(U-1)T.
\]
Then $K(S,T)=K(U,V)$ and $\rho$ acts on $K(U,V)$ by
\[
\rho: \alpha\mapsto -\alpha,~ U\mapsto -U,~ V\mapsto -(U^2+3)/V.
\]

Define $W:=U/\alpha$. We find that
\[
\rho:\alpha\mapsto -\alpha,~ W\mapsto W, ~ T\mapsto -(aW^2+3)/T.
\]

Apply Lemma \ref{l4.1} (3). We obtain $K(x,y)^G=K(W,T)^{\langle
\rho\rangle}$ is $k$-rational because $(a,-a)_k=0$.

\bigskip
\begin{Case}{7} $G\simeq D_4=\langle \sigma,\tau\rangle$. \end{Case}

The non-trivial proper normal subgroups $N$ of $G$ are $\langle -I \rangle$, $\langle -I,\lambda\rangle$, $\langle -I,\tau\rangle$ and $\langle\sigma\rangle$ where $\lambda=\tau\sigma$.

\medskip
Subcase 7.1. $(G,H)=(D_4,\langle -I\rangle)$.

Since $G/H=\langle\sigma,\tau\rangle \simeq V_4$ acts faithfully on $K$,
we may write $K=k(\alpha,\beta)$ with $\alpha$, $\beta$ suitably chosen and $[K:k]=4$.

The group $G$ acts on $K(x,y)$ by
\begin{alignat*}{2}
-I &: \alpha\mapsto \alpha,~\beta\mapsto \beta, ~ x\mapsto\tfrac{1}{x},~ y\mapsto \tfrac{1}{y}, &\q \sigma &: \alpha\mapsto \ol{\alpha},~ \beta\mapsto \beta,~ x\mapsto y,~ y\mapsto \tfrac{1}{x}, \\
\tau &: \alpha\mapsto \alpha,~ \beta\mapsto \ol{\beta},~ x\mapsto y,~ y\mapsto x.
\end{alignat*}
By Lemma \ref{l3.2}, we have $K(x,y)^{\langle -I \rangle}=K(s,t)$ where $s$ and $t$ are given as in \eqref{eq1},
and the actions of $\sigma$ and $\tau$ on $K(s,t)$ are given by
\begin{align*}
\sigma &: \alpha\mapsto \ol{\alpha},~ \beta\mapsto \beta,~ s\mapsto \tfrac{1}{s},~ t\mapsto
\begin{cases} -\frac{1}{t} & \text{if }\fn{char}k\ne 2, \\ \frac{1}{t} & \text{if }\fn{char}k=2, \end{cases} \\
\tau &: \alpha\mapsto \alpha, ~ \beta \mapsto \ol{\beta},~ s\mapsto s,~ t\mapsto
\begin{cases} -t & \text{if }\fn{char}k\ne 2, \\ \frac{1}{t} & \text{if } \fn{char}k=2. \end{cases}
\end{align*}

When $\fn{char}k=2$, $K(x,y)^G=K(s,t)^{\langle
\sigma,\tau\rangle}$ is $k$-rational by Theorem \ref{t1.11}.

When $\fn{char}k\ne 2$, we put
\[
S:=\frac{\alpha(s+1)}{s-1}, \quad T:=\beta t,
\]
then $K(s,t)^{\langle\tau\rangle}=k(\alpha)(S,T)$ and $\sigma$ acts on $k(\alpha)(S,T)$ by
\[
\sigma:\alpha\mapsto \ol{\alpha},~ \beta\mapsto \beta, ~ S\mapsto S,~ T\mapsto\tfrac{-b}{T}.
\]

By Lemma \ref{l4.1} (1),
$K(x,y)^G=k(\alpha)(S,T)^{\langle\sigma\rangle}$ is $k$-rational
if and only if $(a,-b)_k=0$. If $K(x,y)^G$ is not $k$-rational, it
is not $k$-unirational by the last assertion of Lemma \ref{l4.1}.
This completes the proof of the second exceptional case in Theorem
\ref{t1.8}.

\medskip
Subcase 7.2. $(G,H)\simeq (D_4,\langle -I,\lambda\rangle)$,
$N=\langle -I,\lambda\rangle$.

The group $G$ acts on $K(x,y)$ by
\begin{alignat*}{2}
-I &: \alpha\mapsto \alpha,~ x\mapsto \tfrac{1}{x},~ y\mapsto \tfrac{1}{y}, &\q \lambda &: \alpha\mapsto \alpha,~ x\mapsto x,~ y\mapsto \tfrac{1}{y}, \\
\tau &: \alpha\mapsto \ol{\alpha},~ x\mapsto y,~ y\mapsto x
\end{alignat*}
where $K=k(\alpha)$ and $[K:k]=2$.
By Lemma \ref{l3.2}, $K(x,y)^{\langle -I\rangle}=K(s,t)$ where $s$ and $t$ are given as in \eqref{eq1},
and $\lambda$ and $\tau$ act on $K(s,t)$ by
\begin{align*}
\lambda &: \alpha\mapsto\alpha,~ s\mapsto \tfrac{1}{s},~ t\mapsto
\begin{cases} \frac{1}{t} & \text{if }\fn{char} k\ne 2, \\ t & \text{if }\fn{char}k=2, \end{cases} \\
\tau &: \alpha\mapsto \ol{\alpha},~ s\mapsto s,~ t\mapsto
\begin{cases} -t & \text{if } \fn{char}k\ne 2, \\ \frac{1}{t} & \text{if }\fn{char}k=2. \end{cases}
\end{align*}

When $\fn{char}k=2$, the action is the same as that in Subcase
1.3. Hence the result.

When $\fn{char}k\ne 2$, by Lemma \ref{l3.2} again, we have
$K(s,t)^{\langle \lambda \rangle}=K(u,v)$ where $u=(st+1)/(s+t)$
and $v=(st-1)/(s-t)$, and $\tau$ acts on $K(u,v)$ by
\[
\tau: \alpha\mapsto -\alpha,~ u\mapsto -v,~ v\mapsto -u.
\]

Thus $K(x,y)^G=K(u,v)^{\langle\tau\rangle}=k(\alpha(u+v),u-v)$ is
$k$-rational.

\medskip
Subcase 7.3. $(G,H)\simeq (D_4,\langle -I,\lambda \rangle)$.
The group $G$ acts on $K(x,y)$ by
\begin{alignat*}{2}
-I &: \alpha\mapsto \alpha,~ x\mapsto \tfrac{1}{x},~ y\mapsto \tfrac{1}{y}, &\q \tau &: \alpha\mapsto \alpha,~ x\mapsto y,~ y\mapsto x, \\
\sigma &: \alpha\mapsto \ol{\alpha},~ x\mapsto y,~ y\mapsto \tfrac{1}{x}
\end{alignat*}
where $K=k(\alpha)$ and $[K:k]=2$.
By Lemma \ref{l3.2}, $K(x,y)^{\langle -I \rangle}=K(s,t)$ where $s$ and $t$ are given as in \eqref{eq1},
and $\tau$ and $\sigma$ act on $K(s,t)$ by
\begin{align*}
\tau &: \alpha\mapsto\alpha,~ s\mapsto s,~ t\mapsto
\begin{cases}  -t & \text{if }\fn{char} k\ne 2, \\ \frac{1}{t} & \text{if }\fn{char}k=2, \end{cases} \\
\sigma &: \alpha\mapsto \ol{\alpha},~ s\mapsto \tfrac{1}{s},~ t\mapsto
\begin{cases} -\frac{1}{t} & \text{if } \fn{char}k\ne 2, \\ \frac{1}{t} & \text{if }\fn{char}k=2. \end{cases}
\end{align*}

When $\fn{char}k=2$, the action is the same as in Subcase 1.2.
Done.

When $\fn{char}k\ne 2$, $K(s,t)^{\langle \tau\rangle}=K(s,t')$
where $t'=t^2$. It follows that
$K(s,t)^{\langle\tau,\sigma\rangle}=k(s,t')^{\langle\sigma\rangle}$
is $k$-rational.

\medskip
Subcase 7.4. $(G,H)\simeq (D_4,\langle\sigma\rangle)$.
The group $G$ acts on $K(x,y)$ by
\begin{alignat*}{2}
-I &: \alpha\mapsto \alpha,~ x\mapsto \tfrac{1}{x},~ y\mapsto \tfrac{1}{y}, &\q \sigma &: \alpha\mapsto \alpha,~ x\mapsto y,~ y\mapsto \tfrac{1}{x}, \\
\tau &: \alpha\mapsto \ol{\alpha},~ x\mapsto y,~ y\mapsto x
\end{alignat*}
where $K=k(\alpha)$ and $[K:k]=2$.
By Lemma \ref{l3.2}, $K(x,y)^{\langle -I\rangle}=K(s,t)$ where $s$ and $t$ are given as in \eqref{eq1},
and $\tau$ and $\sigma$ act on $K(s,t)$ by
\begin{align*}
\sigma &: \alpha\mapsto\alpha,~ s\mapsto \tfrac{1}{s},~ t\mapsto
\begin{cases}  -\frac{1}{t} & \text{if }\fn{char} k\ne 2, \\ \frac{1}{t} & \text{if }\fn{char}k=2, \end{cases} \\
\tau &: \alpha\mapsto \ol{\alpha},~ s\mapsto s,~ t\mapsto
\begin{cases} -t & \text{if } \fn{char}k\ne 2, \\ \frac{1}{t} & \text{if }\fn{char}k=2. \end{cases}
\end{align*}

When $\fn{char}k=2$, this is the same action as in Subcase 1.1.
Hence $k(x,y)^G$ is $k$-rational.

When $\fn{char}k\ne 2$, by Lemma \ref{l3.1},
$K(s,t)^{\langle\sigma\rangle}=K(u,v)$ where
\[
u=\frac{s-\frac{1}{s}}{st+\frac{1}{st}}, \quad v=\frac{t+\frac{1}{t}}{st+\frac{1}{st}}
\]
and $\tau$ acts on $K(u,v)$ by
\[
\tau:\alpha\mapsto \ol{\alpha},~ u\mapsto -u,~ v\mapsto v.
\]

Thus $K(x,y)^G=K(u,v)^{\langle \tau\rangle}=k(\alpha u,v)$ is
$k$-rational.

\bigskip
\begin{Case}{8} $G\simeq D_6=\langle \rho,\tau\rangle$. \end{Case}

The non-trivial proper normal subgroups $N$ of $G$ are $\langle -I\rangle$, $\langle\rho^2\rangle$, $\langle\rho\rangle$, $\langle\rho^2,\tau\rangle$ and $\langle\rho^2,-\tau\rangle$.

\medskip
Subcase 8.1. $(G,H)\simeq (D_6,\langle -I\rangle)$.

By the same calculation as in Subcase 6.1, $\rho$ and $\tau$ act
on $K(x,y)^{\langle -I\rangle}=K(A,B)$ by
\[
\rho:\alpha\mapsto\ol{\alpha},~ A\mapsto B,~ B \mapsto \tfrac{1}{AB}, \quad
\tau:\beta\mapsto \ol{\beta},~ A\mapsto \tfrac{1}{B},~ B\mapsto \tfrac{1}{A}
\]
where $K=k(\alpha,\beta)$, $[K:k]=6$ and $\langle\rho,\tau\rangle$
acts on $K$ faithfully. Hence $K(x,y)^G=K(A,B)^{\langle
\rho,\tau\rangle}$ is $k$-rational by Theorem \ref{t1.11}.

\medskip
Subcase 8.2. $(G,H)\simeq (D_6,\langle \rho^2\rangle)$.

The group $G$ acts on $K(x,y)$ by
\begin{alignat*}{2}
\rho^2 &: \alpha\mapsto \alpha,~\beta\mapsto\beta,~ x\mapsto y,~ y\mapsto \tfrac{1}{xy}, &\q -\tau &: \alpha\mapsto \ol{\alpha},~\beta\mapsto\beta,~ x\mapsto \tfrac{1}{y},~ y\mapsto \tfrac{1}{x}, \\
\tau &: \alpha\mapsto \alpha,~\beta\mapsto \ol{\beta},~ x\mapsto y,~ y\mapsto x
\end{alignat*}
where $K=k(\alpha,\beta)$ and $[K:k]=4$.
We have $K(x,y)^{\langle \rho^2\rangle}=K(S,T)$ where $S$ and $T$ are defined in Lemma \ref{l3.3} and $-\tau$ and $\tau$ act on $K(S,T)$ by
\begin{align*}
-\tau &: \alpha\mapsto\ol{\alpha},~\beta\mapsto\beta,~ S\mapsto \tfrac{1}{S},~ T\mapsto
\begin{cases}  \frac{T}{S} & \text{if }\fn{char} k\ne 3, \\ T & \text{if }\fn{char}k=3, \end{cases} \\
\tau &: \alpha\mapsto \alpha,~\beta\mapsto \ol{\beta},~ S\mapsto S,~ T\mapsto
\begin{cases} \frac{S(S+\frac{1}{S}-1)}{T} & \text{if } \fn{char}k\ne 3, \\ \frac{1}{T} & \text{if }\fn{char}k=3. \end{cases}
\end{align*}

When $\fn{char}k=3$, $K(x,y)^G=K(S,T)^{\langle -\tau,\tau\rangle}$
is $k$-rational by Theorem \ref{t1.11}.

When $\fn{char}k=2$, $K(x,y)^{\langle -\tau\rangle}=K(U,V)$ where
\[
U=\alpha+\frac{S}{S+1}, \quad V=\frac{T}{S+1},
\]
and $\tau$ acts on $K(U,V)$ by
\[
\tau: \beta\mapsto \ol{\beta},~ U\mapsto U,~ V\mapsto \tfrac{U^2+U+a+1}{V}
\]
where $a=\alpha(\alpha+1)$. By Lemma \ref{l4.1} (4),
$K(x,y)^G=K(U,V)^{\langle\tau\rangle}$ is $k$-rational.

When $\fn{char}k\ne 2,3$, define
\[
U:=\frac{S+1}{S-1}, \quad V:=(U-1)T.
\]

Then $-\tau$ acts on $K(S,T)=K(U,V)$ by
\[
-\tau: \alpha\mapsto -\alpha,~\beta\mapsto\beta,~ U\mapsto -U,~ V\mapsto -V.
\]

Define $P=V/\alpha$ and $Q=U/\alpha$. It follows that
$K(U,V)^{\langle -\tau\rangle}=k(\beta)(P,Q)$ and $\tau$ acts on
$k(\beta)(P,Q)$ by
\[
\tau: \beta\mapsto -\beta,~ P\mapsto \tfrac{Q^2+3/a}{P},~ Q\mapsto Q.
\]

Thus $K(x,y)^G=k(\beta)(P,Q)^{\langle \tau \rangle}$ is
$k$-rational by Lemma \ref{l4.1} (3).

\medskip
Subcase 8.3. $(G,H)\simeq (D_6,\langle\rho \rangle)$.

The group $G$ acts on $K(x,y)$ by
\begin{alignat*}{2}
-I &: \alpha\mapsto \alpha,~ x\mapsto \tfrac{1}{x},~ y\mapsto \tfrac{1}{y}, &\q \rho &: \alpha\mapsto \alpha,~ x\mapsto xy,~ y\mapsto \tfrac{1}{x}, \\
\tau &: \alpha\mapsto \ol{\alpha},~ x\mapsto y,~ y\mapsto x
\end{alignat*}
where $K=k(\alpha)$ and $[K:k]=2$. By the same calculation as in
Subcase 6.1, we have $K(x,y)^{\langle -I\rangle}=K(A,B)$ where $A$
and $B$ are the same given there in Subcase 6.1, and $\rho$ and
$\tau$ act on $K(A,B)$ by
\[
\rho:\alpha\mapsto \alpha,~A\mapsto B,~ B\mapsto \tfrac{1}{AB}, \quad
\tau:\alpha\mapsto \ol{\alpha},~ A\mapsto \tfrac{1}{B},~ B\mapsto \tfrac{1}{A}.
\]

These actions are the same as in the case of $G=S_3^{(1)}$ in Case
5. Thus $K(x,y)^G=K(A,B)^{\langle \rho,\tau\rangle}$ is
$k$-rational.

\medskip
Subcase 8.4. $(G,H)\simeq (D_6,\langle \rho^2,\tau \rangle)$.

The group $G$ acts on $K(x,y)$ by
\begin{alignat*}{2}
\rho^2 &: \alpha\mapsto \alpha,~ x\mapsto y,~ y\mapsto \tfrac{1}{xy}, &\q
\tau &: \alpha\mapsto \alpha,~ x\mapsto y,~ y\mapsto x, \\
\rho &: \alpha\mapsto \ol{\alpha},~ x\mapsto xy,~ y\mapsto \tfrac{1}{x}
\end{alignat*}
where $K=k(\alpha)$ and $[K:k]=2$.
We have $K(x,y)^{\langle \rho^2\rangle}=K(S,T)$ where $S$ and $T$ are defined in Lemma \ref{l3.3}, and $\tau$ and $\rho$ act on $K(S,T)$ by
\begin{align*}
\tau &: \alpha\mapsto\alpha,~ S\mapsto S,~ T\mapsto
\begin{cases}  \frac{S(S+\frac{1}{S}-1)}{T} & \text{if }\fn{char} k\ne 3, \\ \frac{1}{T} & \text{if }\fn{char}k=3, \end{cases} \\
\rho &: \alpha\mapsto \ol{\alpha},~ S\mapsto \frac{1}{S},~ T\mapsto
\begin{cases} \frac{S+\frac{1}{S}-1}{T} & \text{if } \fn{char}k\ne 3, \\ \frac{1}{T} & \text{if }\fn{char}k=3. \end{cases}
\end{align*}

When $\fn{char}k=3$, $K(x,y)^G=K(S,T)^{\langle \tau,\rho
\rangle}=K(S,T+\frac{1}{T})^{\langle\rho\rangle}=k(\alpha(S+1)/(S-1),T+\frac{1}{T})$
is $k$-rational.

When $\fn{char}k\ne 3$, define
\[
U:=S,\quad V:= T+\frac{S(S+\frac{1}{S}-1)}{T}.
\]

Then $K(S,T)^{\langle \tau \rangle}=K(U,V)$ and $\rho$ acts on
$K(U,V)$ by
\[
\rho: \alpha \mapsto \ol{\alpha},~ U\mapsto \tfrac{1}{U},~ V\mapsto \tfrac{V}{U}.
\]

Thus $K(x,y)^G=K(U,V)^{\langle \rho \rangle}$ is $k$-rational by
Theorem \ref{t1.11}.

\medskip
Subcase 8.5. $(G,H)=(D_6,\langle \rho^2,-\tau\rangle)$.

The group $G$ acts on $K(x,y)$ by
\begin{alignat*}{2}
\rho^2 &: \alpha\mapsto \alpha,~ x\mapsto y,~ y\mapsto \tfrac{1}{xy}, &\q
-\tau &: \alpha\mapsto \alpha,~ x\mapsto \tfrac{1}{y},~ y\mapsto \tfrac{1}{x}, \\
\rho &: \alpha\mapsto \ol{\alpha},~ x\mapsto xy,~ y\mapsto \tfrac{1}{x}
\end{alignat*}
where $K=k(\alpha)$ and $[K:k]=2$.
We have $K(x,y)^{\langle \rho^2\rangle}=K(S,T)$ where $S$ and $T$ are defined in Lemma \ref{l3.3} and the actions of $\tau$ and $\rho$ on $K(S,T)$ are given by
\begin{align*}
-\tau &: \alpha\mapsto\alpha,~ S\mapsto \tfrac{1}{S},~ T\mapsto
\begin{cases}  \frac{T}{S} & \text{if }\fn{char} k\ne 3, \\ T & \text{if }\fn{char}k=3, \end{cases} \\
\rho &: \alpha\mapsto \ol{\alpha},~ S\mapsto \tfrac{1}{S},~ T\mapsto
\begin{cases} \frac{S+\frac{1}{S}-1}{T} & \text{if } \fn{char}k\ne 3, \\ \frac{1}{T} & \text{if }\fn{char}k=3. \end{cases}
\end{align*}

If $\fn{char}k=3$, then $K(x,y)^{\langle \rho^2,-\tau,\rho
\rangle}=K(S,T)^{\langle -\tau,\rho
\rangle}=K(S+\frac{1}{S},T)^{\langle\rho\rangle}=k(S+\frac{1}{S},\alpha(T+1)/(T-1))$
is $k$-rational.

If $\fn{char}k\ne 3$, define
\[
U:=\frac{T(S+1)}{S}, \quad V:=\frac{T^2}{S},
\]
then $K(S,T)^{\langle -\tau\rangle}=K(U,V)$ and the action of $\rho$ on $K(U,V)$ is given by
\[
\rho:\alpha \mapsto \ol{\alpha},~ U\mapsto \tfrac{U(U^2-3V)}{V^2},~ V\mapsto \tfrac{(U^2-3V)^2}{V^3}.
\]

Define $W:=(U^2-3V)/V$. We find that
\[
\rho:\alpha \mapsto \ol{\alpha},~ U\mapsto \tfrac{W^2+3W}{U},~ W\mapsto W.
\]

Apply Lemma \ref{l4.1} (3), (4). We find $K(x,y)^G=K(U,W)^{\langle
\rho \rangle}$ is $k$-rational.
\end{proof}

Saltman discussed the relationship of $K_{\alpha}(M)^G$ and the
embedding problem in [\cite{Sa1}]; in particular, see Section 3 of
[\cite{Sa1}]. In the following, we reformulate the two exceptional
cases of Theorem \ref{t1.8} in terms of the embedding problem.

\begin{prop} \label{p4.2}
{\rm (1)}
Let $k$ be a field with $\fn{char}k\ne 2$, $a\in k\backslash k^2$ and $K=k(\sqrt{a})$.
Let $G=\langle \sigma\rangle$ act on $K(x,y)$ by
\[
\sigma: \sqrt{a} \mapsto -\sqrt{a},~x\mapsto \tfrac{1}{x},~ y\mapsto -\tfrac{1}{y}.
\]

Then $K(x,y)^G$ is rational over $k$ if and only if the quadratic extension $k(\sqrt{a})/k$ can be embedded into a $C_4$-extension of $k$,
i.e.\ there is a Galois extension $L/k$ such that $k\subset k(\sqrt{a})\subset L$ and $\fn{Gal}(L/k)\simeq C_4$.

{\rm (2)}
Let $k$ be a field with $\fn{char}k\ne 2$, $a,b\in k\backslash k^2$ such that $[K:k]=4$ where $K=k(\sqrt{a},\sqrt{b})$.
Let $G=\langle \sigma,\tau \rangle\simeq D_4$ act on $K(x,y)$ by
\begin{align*}
\sigma &: \sqrt{a}\mapsto -\sqrt{a},~ \sqrt{b}\mapsto\sqrt{b},~ x\mapsto \tfrac{1}{x},~y\mapsto -\tfrac{1}{y}, \\
\tau &: \sqrt{a}\mapsto \sqrt{a},~ \sqrt{b}\mapsto -\sqrt{b},~ x\mapsto x,~ y\mapsto -y.
\end{align*}

Then $K(x,y)^G$ is rational over $k$ if and only if there is a Galois extension $L/k$ such that $k\subset K \subset L$ and $\fn{Gal}(L/k)\simeq D_4$.
\end{prop}

\begin{proof}
Check the proof of Case 3 and Subcase 7.1 of Theorem \ref{t1.8},
and use the well-known results of the embedding problem. For
example, for the obstruction to the $1 \to C_2 \to C_4 \to C_2 \to
1$, see [\cite{La}, Exercise 8, page 217; \cite{Ki}, page 837;
\cite{Le}, page 37]; for the obstruction to the $1 \to C_2 \to D_4
\to C_2 \times C_2 \to 1$, see [\cite{Ki}, page 840; \cite{Le},
page 38].
\end{proof}

\begin{example} \label{ex4.3}
Let $k$ be a field of $\fn{char}k\ne 2$, $K=k(\alpha,\beta)$ be a
biquadratic extension of $k$. We consider the following actions of
$k$-automorphisms of $D_4=\langle\sigma,\tau \rangle$ on $K(x,y)$:
\begin{align*}
\sigma_\alpha &: \alpha \mapsto -\alpha,~ \beta\mapsto \beta,~ x\mapsto y,~ y\mapsto\tfrac{1}{x}, \\
\sigma_\beta &: \alpha\mapsto \alpha,~ \beta\mapsto -\beta,~ x\mapsto y,~ y\mapsto \tfrac{1}{x}, \\
\sigma_{\alpha\beta} &: \alpha \mapsto -\alpha,~ \beta\mapsto -\beta,~ x\mapsto y,~ y\mapsto \tfrac{1}{x}, \\
\tau_\alpha &: \alpha \mapsto -\alpha,~ \beta\mapsto \beta,~ x\mapsto y,~ y\mapsto x, \\
\tau_\beta &: \alpha\mapsto \alpha,~ \beta\mapsto -\beta,~ x\mapsto y,~ y\mapsto x, \\
\tau_{\alpha\beta} &: \alpha\mapsto -\alpha,~ \beta\mapsto -\beta,~ x\mapsto y,~ y\mapsto x.
\end{align*}

Define $L_{\alpha,\beta}=K(x,y)^{\langle
\sigma_\alpha,\tau_\beta\rangle}$ where $a=\alpha^2$ and
$b=\beta^2$. It is not difficult to verify that
\begin{enumerate}
\item[(1)] $L_{\alpha,\beta}$ is rational over $k$ iff
$L_{\alpha,\alpha\beta}$ is rational over $k$ iff $(a,-b)_k=0$,
\item[(2)] $L_{\beta,\alpha}$ is rational over $k$ iff
$L_{\beta,\alpha\beta}$ is rational over $k$ iff $(b,-a)_k=0$,
\item[(3)] $L_{\alpha\beta,\alpha}$ is rational over $k$ iff
$L_{\alpha\beta,\beta}$ is rational over $k$ iff $(a,b)_k=0$.
\end{enumerate}

In particular, if $\sqrt{-1}\in k$ then the rationality of the six
fixed fields $L_{\alpha,\beta}$, $L_{\alpha,\alpha\beta}$,
$L_{\beta,\alpha}$, $L_{\beta,\alpha\beta}$,
$L_{\alpha\beta,\alpha}$ and $L_{\alpha\beta,\beta}$ over $k$
coincide.

On the other hand, consider the case $k=\bm{Q}$,
$K=\bm{Q}(\alpha,\beta)$ where $\alpha^2=-1$ and $\beta^2=p$ where
$p$ is a prime number with $p\equiv 1$ (mod 4). Then
$L_{\alpha,\beta}$ and $L_{\alpha,\alpha\beta}$ are not rational
over $\bm{Q}$ because $(-1,-p)_{\bm{Q}}=(-1,-1)_{\bm{Q}} \neq 0$,
while $L_{\beta,\alpha}$, $L_{\beta,\alpha\beta}$,
$L_{\alpha\beta,\alpha}$ and $L_{\alpha\beta,\beta}$ are rational
over $\bm{Q}$.
\end{example}

\section{Applications of Theorem \ref{t1.8}}

\begin{lemma}[{[\cite{AHK}, Theorem 3.1]}] \label{l5.1}
Let $L$ be any field, $L(x)$ be the rational function field of one
variable over $L$, and $G$ be a finite group acting on $L(x)$.
Suppose that, for any $\sigma\in G$, $\sigma(L)\subset L$ and
$\sigma(x)=a_\sigma x+b_\sigma$ where $a_\sigma, b_\sigma \in L$
and $a_\sigma \ne 0$. Then $L(x)^G=L^G(f)$ for some polynomial
$f\in L[x]^G$. In fact, if $m=\min \{\deg g(x): g(x)\in
L[x]^G\backslash L\}$, any polynomial $f\in L[x]^G$ with $\deg
f=m$ satisfies the property $L(x)^G=L^G(f)$.
\end{lemma}

\bigskip
\begin{proof}[\indent Proof of Theorem \ref{t1.10}] ~

Consider first the case $M=M_1\oplus M_2$ where $\fn{rank}_{\bm{Z}} M_1=3$ and $\fn{rank}_{\bm{Z}} M_2=1$.

Write $M_1=\bigoplus_{1\le i\le 3} \bm{Z}\cdot x_i$, $M_2=\bm{Z}\cdot y$, and $k(M)=k(x_1,x_2,x_3,y)$.

Since $G$ acts on $k(M)$ by purely monomial $k$-automorphisms,
it follows that $\sigma(y)=y$ or $\frac{1}{y}$ for any $\sigma\in G$.
Define $z=\frac{y-1}{y+1}$ (if $\fn{char}k\ne 2$),
and $z=\frac{1}{y+1}$ (if $\fn{char}k=2$).
It follows that $\sigma(z)=z$ or $-z$ (if $\fn{char}k\ne 2$),
and $\sigma(z)=z$ or $z+1$ (if $\fn{char}k=2$).
Apply Lemma \ref{l5.1}.
We find that $k(M)^G=k(x_1,x_2,x_3)^G(f)$ for some polynomial $f\in k(x_1,x_2,x_3)[z]^G$.

Since $G$ acts on $k(x_1,x_2,x_3)^G$ by purely monomial
$k$-automorphisms, the fixed field $k(x_1,x_2,x_3)^G$ is
$k$-rational by Theorem \ref{t1.14}. Hence the result.

From now on we will concentrate on the case $M=M_1\oplus M_2$ with
$\fn{rank}_{\bm{Z}}M_1=\fn{rank}_{\bm{Z}} M_2=2$.

\bigskip
Step 1. Write $M_1=\bigoplus_{1\le i\le 2} \bm{Z}\cdot x_i$,
$M_2=\bigoplus_{1\le i\le 2} \bm{Z}\cdot y_i$, and
$k(M)=k(x_1,x_2,y_1,y_2)$.

We will apply Theorem \ref{t1.8} to show that, for most
situations, $k(M)$ is $k$-rational. In fact, regarding
$k(x_1,x_2)$ as the field $K$, we will use Theorem \ref{t1.8} to
see whether $K(y_1,y_2)^G$ is $k$-rational. Alternatively, we may
regard $k(y_1,y_2)$ as the field $K'$ and try to study
$K'(x_1,x_2)^G$.

In both situations, $G$ acts on $K(y_1,y_2)$ and $K'(x_1,x_2)$ by
purely quasi-monomial $k$-automorphisms. Apply Theorem \ref{t1.8}.
If we are lucky, then either $K(y_1,y_2)^G$ is rational over
$K^G=k(x_1,x_2)^G$ or $K'(x_1,x_2)^G$ is rational over
$K'^G=k(y_1,y_2)^G$. Since $k(x_1,x_2)^G$ and $k(y_1,y_2)^G$ are
$k$-rational by Theorem \ref{t1.13}, it follows that either
$K(y_1,y_2)^G$ or $K'(x_1,x_2)^G$ is $k$-rational.

In summary, $k(M)^G$ is $k$-rational except that both
$K(y_1,y_2)^G$ and $K'(x_1,x_2)^G$ fall into the exceptional
situation of Theorem \ref{t1.8}.

\medskip
Step 2. Suppose that $K(y_1,y_2)^G$ and $K'(x_1,x_2)^G$ fall into
the exceptional situation of Theorem \ref{t1.8}. In particular,
$\fn{char}k \ne 2$.

Without loss of generality, we may assume $M$ is a faithful
$G$-lattice, i.e.\ the subgroup $\{\sigma\in G: \sigma(x_i)=x_i$,
$\sigma(y_i)=y_i$ for $1\le i\le 2\}=\{1\}$.

Define $N_1=\{\sigma\in G:\sigma(x_i)=x_i$ for $1\le i\le 2\}$,
$N_2=\{\sigma\in G:\sigma(y_i)=y_i$ for $1\le i\le 2\}$. Note that
$N_1\triangleleft G$, $N_2\triangleleft G$, and $N_1\cap
N_2=\{1\}$.

In the field $K(y_1,y_2)$, $N_2$ is the subgroup $N$ of Theorem
\ref{t1.8} while $N_1$ is the subgroup $H$ of Theorem \ref{t1.8}.
The reader may interpret $N_1$ and $N_2$ for the field
$K'(x_1,x_2)$ himself.

Since $K(y_1,y_2)^G$ and $K'(x_1,x_2)^G$ fall into the exceptional
situation, we get (i) $(G/N_2,N_1N_2/N_2)\simeq (C_4,C_2)$ or
$(D_4,C_2)$, and (ii) $(G/N_1,N_1N_2/N_1)\simeq (C_4,C_2)$ or
$(D_4,C_2)$.

Note that $N_1\simeq N_1N_2/N_2$ and $N_2\simeq N_1N_2/N_1$. Hence
$N_1 \simeq N_2\simeq C_2$. It follows that either $G/N_1\simeq
C_4\simeq G/N_2$ or $G/N_1\simeq D_4\simeq G/N_2$; the possibility
that $(G/N_1,G/N_2)\simeq (C_4,D_4)$ or $(D_4,C_4)$ is ruled out
because $|G/N_1|=|G/N_2|$.

\medskip
Step 3. We consider the cases $(G/N_1,G/N_2)\simeq (C_4,C_4)$ and
$(D_4,D_4)$ separately.

\medskip
\begin{Case}{1} $G/N_1\simeq C_4$ and $G/N_2 \simeq C_4$. \end{Case}

Write $N_1=\langle \tau_1\rangle$, $N_2=\langle \tau_2\rangle$.

Note that $\tau_2$ and $G/N_1$ act faithfully on $M_1$. By Theorem
\ref{t3.1}, $G/N_1=\langle\sigma_1\rangle \simeq C_4$ and we may
choose the generators $x_1$, $x_2$ of $k(M_1)$ satisfying
$\sigma_1:x_1\mapsto x_2$, $x_2\mapsto\frac{1}{x_1}$,
$\tau_2=\sigma_1^2: x_1\mapsto\frac{1}{x_1}$,
$x_2\mapsto\frac{1}{x_2}$ (also see Case 3 in the proof of Theorem
\ref{t1.8}).

Since $\tau_2$ acts trivially on $k(M_2)$, we find that
$k(M)^{\langle\tau_2\rangle}=k(x_1,x_2)^{\langle\tau_2\rangle}(y_1,y_2)=k(s,t)(y_1,y_2)$
where $s$, $t$ are defined as
\begin{equation}
s=\frac{x_1x_2+1}{x_1+x_2},\quad t=\frac{x_1x_2-1}{x_1-x_2}
\label{eq10}
\end{equation}
by Lemma \ref{l3.2}. Moreover, $\sigma_1(s)=\frac{1}{s}$,
$\sigma_1(t)=-\frac{1}{t}$ (see Case 3 in the proof of Theorem
\ref{t1.8}).

On the other hand, $G/\langle\tau_2\rangle \simeq G/N_2 \simeq
C_4$ acts faithfully on $k(M_2)$. We may choose the generators
$y_1$, $y_2$ and $G/N_2=\langle\sigma_2\rangle$ such that
$\sigma_2:y_1\mapsto y_2\mapsto \frac{1}{y_1}$ by Theorem
\ref{t3.1} again.

Now
$k(M)^G=\{k(x_1,x_2)(y_1,y_2)^{N_2}\}^{G/N_2}=k(s,t)(y_1,y_2)^{\langle\sigma_2\rangle}$.

The action of $\sigma_2$ on $k(s,t)$ is induced by the action of
$G$ on $k(s,t)$. But the action of $G$ on $k(s,t)$ is just that of
$\sigma_1$. Hence $\sigma_2(s)=\frac{1}{s}$ and
$\sigma_2(t)=-\frac{1}{t}$.

Define $u=\frac{s+1}{s-1}$. Then $\sigma_2(u)=-u$. Hence
$k(s,t)(y_1,y_2)^{\langle\sigma_2\rangle}=k(y_1,y_2,t)(u)^{\langle\sigma_2\rangle}=k(y_1,y_2,t)^{\langle\sigma_2\rangle}(f)$
for some polynomial $f$ by Lemma \ref{l5.1}.

Note that $\sigma_2^2(t)=t$ and $\sigma_2^2(y_i)=\frac{1}{y_i}$
for $i=1,2$. Define $z_1=\frac{y_1y_2+1}{y_1+y_2},
z_2=\frac{y_1y_2-1}{y_1-y_2}$. By Lemma \ref{l3.2}
$k(y_1,y_2,t)^{\langle\sigma_2^2 \rangle}=k(z_1,z_2,t)$ and
$\sigma_2(z_1)=\frac{1}{z_1}$, $\sigma_2(z_2)=\frac{-1}{z_2}$.

Regard $k(z_1,z_2,t)^{\langle\sigma_2
\rangle}=k(z_2,t)(z_1)^{\langle\sigma_2 \rangle}$ as the function
field of a $1$-dimensional algebraic torus over the field
$k(z_2,t)^{\langle\sigma_2 \rangle}$. We find that
$k(z_1,z_2,t)^{\langle\sigma_2 \rangle}$ is rational over
$k(z_2,t)^{\langle\sigma_2 \rangle}$ (see Example \ref{ex1.2}). On
the other hand, the field $k(z_2,t)^{\langle\sigma_2 \rangle}$ is
$k$-rational by Theorem \ref{t1.13}. Hence the result.

\medskip
\begin{Case}{2} $G/N_1\simeq D_4$ and $G/N_2 \simeq D_4$. \end{Case}

The proof is almost the same as Case 1.

Write $N_1=\langle\tau_1\rangle$, $N_2=\langle\tau_2\rangle$,
$G/N_1=\langle\sigma_1,\lambda_1\rangle \simeq D_4$,
$G/N_2=\langle\sigma_2,\lambda_2\rangle \simeq D_4$.

As before, apply Theorem \ref{t3.1}. We may assume that
$\tau_2=\sigma_1^2$ and
\begin{alignat*}{2}
\sigma_1 &: x_1\mapsto x_2\mapsto \tfrac{1}{x_1}; &\q \lambda_1 &: x_1\leftrightarrow x_2; \\
\sigma_2 &: y_1\mapsto y_2\mapsto \tfrac{1}{y_1}; & \lambda_2 &:
y_1\leftrightarrow y_2.
\end{alignat*}

It follows that $k(x_1,x_2)^{\langle\tau_2\rangle}=k(s,t)$ where
$s$ and $t$ are defined by \eqref{eq10}. Moreover, as in Subcase
7.1 in the proof of Theorem \ref{t1.8}, we have
\[
\sigma_1: s\mapsto \tfrac{1}{s},~ t\mapsto -\tfrac{1}{t}; \quad
\lambda_1: s\mapsto s,~ t\mapsto -t.
\]

Note that
$k(M)^G=\{k(x_1,x_2)(y_1,y_2)^{N_2}\}^{G/N_2}=k(s,t)(y_1,y_2)^{G/N_2}$.
The action of $G/N_2$ on $k(s,t)$ is the action induced by
$G/N_1=\langle\sigma_1,\lambda_1\rangle$ on $k(s,t)$. In
particular, it is a purely monomial action on $k(s)$.

Define $u=\frac{s+1}{s-1}$. The rationality problem of
$k(s,t)(y_1,y_2)^{G/N_2}$ is reduced to that of
$k(y_1,y_2,t)^{G/N_2}$ by Lemma \ref{l5.1}. The proof that
$k(y_1,y_2,t)^{G/N_2}$ is $k$-rational is almost the same as that
of the last stage of Case 1. The details of the proof is omitted.
\end{proof}

\begin{prop} \label{p5.2}
Let $G$ be a finite group, $M$ a $G$-lattice. Assume that
$M=M_1\oplus M_2\oplus M_3$ as $\bm{Z}[G]$-modules where
$\fn{rank}_{\bm{Z}} M_1=\fn{rank}_{\bm{Z}} M_2=2$ and
$\fn{rank}_{\bm{Z}}M_3=1$. For any field $k$, suppose that $G$
acts on $k(M)$ by purely monomial $k$-automorphisms. Then $k(M)^G$
is $k$-rational.
\end{prop}

\begin{proof}
Write $M_3=\bm{Z}\cdot z$.
Then $k(M)=k(M_1\oplus M_2)(z)$ and $\sigma(z)=z$ or $\frac{1}{z}$ for any $\sigma \in G$.
Define $u=\frac{z+1}{z-1}$ (if $\fn{char}k\ne 2$),
and $u=\frac{1}{z+1}$ (if $\fn{char}k=2$).
It follows that $\sigma(u)=\pm u$ (if $\fn{char}k\ne 2$), and $\sigma(u)=u$ or $u+1$ (if $\fn{char}k=2$).

Apply Lemma \ref{l5.1}.
We find $k(M)^G=k(M_1\oplus M_2)^G(f)$ for some polynomial $f$.
Since $k(M_1\oplus M_2)^G$ is $k$-rational by Theorem \ref{t1.10},
it follows that $k(M)^G$ is $k$-rational.
\end{proof}

\section{Some decomposable lattices of rank 5}

Before extending the method in the proof of Theorem \ref{t1.10}, we recall the definition of retract rationality.

\begin{defn}[{[\cite{SaA,Ka4}]}] \label{d6.1}
Let $k$ be an infinite field and $L/k$ be a field extension. $L$
is called retract $k$-rational if $L$ is the quotient field of
some integral domain $A$ where $k\subset A\subset L$ satisfying
the conditions that there exist a polynomial ring
$k[X_1,\ldots,X_m]$, some non-zero polynomial $f\in
k[X_1,\ldots,X_n]$, and $k$-algebra morphisms $\varphi:A\to
k[X_1,\ldots,X_m][1/f]$, $\psi: k[X_1,\ldots,X_m][1/f]\to A$ such
that $\psi\circ\varphi =1_A$.
\end{defn}

It is not difficult to see that
``$k$-rational"$\Rightarrow$``stably
$k$-rational"$\Rightarrow$``retract $k$-rational"
$\Rightarrow$``$k$-unirational".

\begin{theorem} \label{t6.2}
Let $G$ be a finite group, $M$ be a $G$-lattice. Assume that {\rm
(i)} $M=M_1\oplus M_2$ as $\bm{Z}[G]$-modules where
$\fn{rank}_{\bm{Z}}M_1=3$ and $\fn{rank}_{\bm{Z}}M_2=2$, {\rm
(ii)} either $M_1$ or $M_2$ is a faithful $G$-lattice. Let $k$ be
a field and $G$ act on $k(M)$ by purely monomial
$k$-automorphisms. Then $k(M)^G$ is $k$-rational except the
following situation \rm{:} $\fn{char}k\ne 2$, $G=\langle
\sigma,\tau\rangle \simeq D_4$ and $M_1=\bigoplus_{1\le i\le 3}
\bm{Z} x_i$, $M_2=\bigoplus_{1\le j\le 2} \bm{Z}y_j$ such that
$\sigma:x_1\leftrightarrow x_2$, $x_3\mapsto -x_1-x_2-x_3$,
$y_1\mapsto y_2\mapsto -y_1$, $\tau: x_1\leftrightarrow x_3$,
$x_2\mapsto -x_1-x_2-x_3$, $y_1\leftrightarrow y_2$ where the
$\bm{Z}[G]$-module structure of $M$ is written additively.

For the exceptional case, $k(M)^G$ is not retract $k$-rational. In
particular, it is not $k$-rational.
\end{theorem}

\begin{proof}
Consider first the case $M_1$ is a faithful $G$-lattice. It
follows that $k(M)^G=k(M_1)(M_2)^G$ is isomorphic to the function
field of some 2-dimensional algebraic torus defined over
$k(M_1)^G$ and split over $k(M_1)$ (see Example \ref{ex1.2}). By
Theorem \ref{t1.11}, $k(M)^G$ is rational over $k(M_1)^G$. On the
other hand, $k(M_1)^G$ is rational over $k$ by Theorem
\ref{t1.14}. Hence $k(M)^G$ is $k$-rational.

For the rest of the proof we consider the case when $M_2$ is a
faithful $G$-lattice.

\medskip
Step 1. Write $M_1=\bigoplus_{1\le i\le 3} \bm{Z} x_i$,
$M_2=\bigoplus_{1\le j\le 2} \bm{Z} y_j$. Define $H=\{\sigma \in
G: \sigma(x_i)=x_i$ for $1\le i\le 3\}$.

There are three possibilities : $H=\{1\}$, $G$ or a non-trivial
proper normal subgroup of $G$.

If $H=\{1\}$, then $M_1$ is also a faithful $G$-lattice.
Hence $k(M)^G$ is $k$-rational by the first paragraph of this proof.

If $H=G$, then
$k(M)^G=k(y_1,y_2)(x_1,x_2,x_3)^G=k(y_1,y_2)^G(x_1,x_2,x_3)$.
Since $k(y_1,y_2)^G$ is $k$-rational by Theorem \ref{t1.13}, it
follows that $k(M)^G$ is $k$-rational.

If $H$ is a non-trivial proper normal subgroup of $G$, we apply
Theorem \ref{t1.8} to $k(M)=k(M_1)(M_2)=K(y_1,y_2)$ with
$K=k(M_1)$. It follows that $k(M)^G$ is rational over $k(M_1)^G$
except when $\fn{char}k\ne 2$ and $(G,H)\simeq (C_4,C_2)$ or
$(D_4,C_2)$. Since $k(M_1)^G$ is $k$-rational by Theorem
\ref{t1.14}, it follows that $k(M)^G$ is $k$-rational except for
the above situation.

\medskip
Step 2.
Consider the exceptional situation in Theorem \ref{t1.8},
i.e.\ $\fn{char}k\ne 2$ and $(G,H)\simeq (C_4,C_2)$, $(D_4,C_2)$.

We consider the case $(G,H)\simeq (C_4,C_2)$ first.
The case $(G,H)\simeq (D_4,C_2)$ will be treated later.

Write $G=\langle\sigma \rangle \simeq C_4$ and $H=\langle\sigma^2\rangle$.

Note that
$k(M)^G=k(M_2)(M_1)^G=\{k(M_2)(M_1)^H\}^{G/H}=\{k(M_2)^H(M_1)\}^{G/H}$.
Since $G/H$ acts faithfully on $k(M_2)^H$, we may regard
$\{k(M_2)^H(M_1)\}^{G/H}$ as the function field of a 3-dimensional
algebraic torus defined over $k(M_2)^H$.

We try to apply Theorem \ref{t1.12} to assert that
$\{k(M_2)^H(M_1)\}^{G/H}$ is rational over $\{k(M_2)^H\}^{G/H}$.
For this purpose, we should study Kunyavskii's list [\cite{Ku},
Theorem 1] closely. Since $G/H \simeq C_2$, every algebraic torus
split over a $C_2$-extension is rational by [\cite{Ku}, Theorem 1]
(this fact may also be observed by the integral extension of
$C_2$). Hence $\{k(M_2)^H(M_1)\}^{G/H}$ is rational over
$\{k(M_2)^H\}^{G/H}\simeq k(M_2)^G$. Since $k(M_2)^G$ is
$k$-rational by Theorem \ref{t1.13}, it follows that $k(M)^G$ is
$k$-rational.

\medskip
Step 3.
Now consider the case $\fn{char}k\ne 2$ and $(G,H)\simeq (D_4,C_2)$.

From Subcase 7.1 in the proof of Theorem \ref{t1.8} in Section 4,
it is not difficult to see that $G=\langle\sigma,\tau\rangle
\simeq D_4$ with $\sigma: y_1\mapsto y_2\mapsto \frac{1}{y_1}$,
$\tau: y_1\leftrightarrow y_2$. Note that
$H=\langle\sigma^2\rangle$ and $G/H\simeq C_2\times C_2$ acts
faithfully on $M_1$.

Note that $k(M)^G=k(M_2)(M_1)^G$. We may regard $k(M_2)(M_1)^G$ as
the function field of an algebraic torus defined over $k(M_2)^G$.

We check Kunyavskii's list [\cite{Ku}, Theorem 1] again. Such an
algebraic torus is rational over $k(M_2)^G$ (and therefore
$k(M)^G$ is $k$-rational), except when $M_1$ is the lattice
$U_1$(see [\cite{Ku}, page 2]). The lattice $U_1$ is a faithful
lattice over $\bm{Z}[G/H]$ with $G/H \simeq C_2\times C_2$ .

The lattice $U_1$ or the associated finite subgroup in
$GL_3(\bm{Z})$ is conjugate to the group $G_{3,1,4}$ in
[\cite{BBNWZ}]. Hence the action of
$G/H=\langle\sigma,\tau\rangle$ on $x_1$, $x_2$, $x_3$ may be
given as $\sigma:x_1\leftrightarrow x_2$,
$x_3\mapsto\frac{1}{x_1x_2x_3}$, $\tau: x_1\leftrightarrow x_3$,
$x_2\mapsto\frac{1}{x_1x_2x_3}$.

We conclude that $k(M)^G=k(M_2)(M_1)^G$ is $k$-rational except
when $G/H=\langle\sigma,\tau\rangle$ acts on $M_1$ as $\sigma:
x_1\leftrightarrow x_2$, $x_3\mapsto -x_1-x_2-x_3$, $\tau:
x_1\leftrightarrow x_3$, $x_2\mapsto -x_1-x_2-x_3$.

It follows that $M_1$ is a $G$-lattice where
$G=\langle\sigma_0,\tau\rangle \simeq D_4$ and $\sigma_0^2$ acting
trivially on $M_1$ (note that $\sigma$ is the image of $\sigma_0$
in $G/H$).

Re-write the generators of $G$. We conclude that the only
unsettled situation is the following: $\fn{char}k \ne 2$,
$G=\langle\sigma,\tau\rangle \simeq D_4$, $\sigma:
x_1\leftrightarrow x_2$, $x_3\mapsto -x_1-x_2-x_3$, $y_1\mapsto
y_2\mapsto -y_1$, $\tau: x_1\leftrightarrow x_3$, $x_2\mapsto
-x_1-x_2-x_3$, $y_1\leftrightarrow y_2$.

For this situation, we will show that $k(x_1,x_2,x_3,y_1,y_2)^G$
is not $k$-rational; in fact, it is not even retract $k$-rational.

Note that $k(x_1,x_2,x_3,y_1,y_2)^G=(k(y_1,y_2)^{\langle \sigma^2
\rangle}(x_1,x_2,x_3))^{G/\langle \sigma^2 \rangle}$. The
$3$-dimensional algebraic torus with function field
$(k(y_1,y_2)^{\langle \sigma^2 \rangle}(x_1,x_2,x_3))^{G/\langle
\sigma^2 \rangle}$ is not rational over $(k(y_1,y_2)^{\langle
\sigma^2 \rangle})^{G/\langle \sigma^2 \rangle}=k(y_1,y_2)^G$ by
Kunyavskii's Theorem [\cite{Ku}, Theorem 1]. But this doesn't
entail that $k(x_1,x_2,x_3,y_1,y_2)^G$ is or is not rational over
$k$. The proof of the non-rationality of
$k(x_1,x_2,x_3,y_1,y_2)^G$ over $k$ will be given in the following
Theorem \ref{t6.4}.
\end{proof}

\begin{theorem} \label{t6.3}
Let $k$ be an infinite field with $\fn{char}k\ne 2$, and
$G=\langle\tau\rangle \simeq C_2$ act on the rational function
field $k(x_1,x_2,x_3,x_4)$ by $k$-automorphisms defined as
\[
\tau: x_1\mapsto -x_1,~ x_2\mapsto \tfrac{x_4}{x_2},~ x_3\mapsto \tfrac{-(x_4-1)x_1^2+x_4(x_4-1)}{x_3},~ x_4\mapsto x_4.
\]

Then $k(x_1,x_2,x_3,x_4)^{\langle \tau\rangle}$ is not retract $k$-rational.
\end{theorem}

\begin{proof}
Suppose that $k(x_1,x_2,x_3,x_4)^{\langle \tau\rangle}$ is retract $k$-rational and $\bar{k}$ is an algebraic closure of $k$.
Choose the $k$-algebra $A$ and morphisms $\varphi$, $\psi$ provided in Definition \ref{d6.1}.
We get $\psi\circ \varphi =1$ where
\[
A \xrightarrow{~\varphi~} k[X_1,\ldots,X_m][1/f] \xrightarrow{~\psi~} A.
\]

Tensor the above morphisms with $\bar{k}$. We get $A\otimes_k
\bar{k} \to \bar{k} [X_1,\ldots,X_m][1/f] \to A\otimes_k \bar{k}$
and the composite map is $1_{A\otimes_k \bar{k}}$. In other words,
$A\otimes_k \bar{k}$ is also retract $\bar{k}$-rational.

Thus, to prove Theorem \ref{t6.3}, we may assume that the field $k$ is algebraically closed.

\medskip
Step 1.
Let $k$ be an algebraically closed field with $\fn{char}k\ne 2$,
and $H=\langle \sigma,\lambda \rangle \simeq C_2\times C_2$ act on the rational function field $k(v_1,v_2,v_3,v_4)$ by $k$-automorphisms defined as
\begin{align*}
\sigma &: v_1\mapsto v_2v_4,~ v_2\mapsto v_1v_4,~ v_3\mapsto \tfrac{1}{v_3},~ v_4\mapsto \tfrac{1}{v_4}, \\
\lambda &: v_1 \mapsto \tfrac{1}{v_1v_4},~ v_2\mapsto \tfrac{1}{v_2v_4},~ v_3\mapsto -v_3,~ v_4\mapsto v_4.
\end{align*}

By [\cite{CHKK}, Example 5.11, page 2355],
the fixed field $k(v_1,v_2,v_3,v_4)^H$ is not retract $k$-rational.
We will show that $k(v_1,v_2,v_3,v_4)^H$ is $k$-isomorphic to $k(x_1,x_2,x_3,x_4)^{\langle \tau\rangle}$ and finish the proof.

\medskip
Step 2.
Define
\[
u_1=v_1,\quad u_2=v_2v_4,\quad u_3=\frac{v_3+1}{v_3-1},\quad u_4=\frac{v_4+1}{v_4-1}.
\]
Then $k(v_1,v_2,v_3,v_4)=k(u_1,u_2,u_3,u_4)$ and
\begin{align*}
\sigma &: u_1\mapsto u_2,~ u_2\mapsto u_1,~ u_3\mapsto -u_3,~ u_4\mapsto -u_4, \\
\lambda &: u_1\mapsto \tfrac{u_4-1}{u_1(u_4+1)},~ u_2\mapsto \tfrac{u_4+1}{u_2(u_4-1)},~ u_3\mapsto \tfrac{1}{u_3},~ u_4\mapsto u_4.
\end{align*}
It follows that $k(u_1,u_2,u_3,u_4)^{\langle \sigma\rangle}=k(w_1,w_2,w_3,w_4)$ where
\[
w_1=\frac{u_1+u_2}{2},\quad w_2=\frac{u_1-u_2}{2u_4}, \quad w_3=\frac{u_3}{u_4}, \quad w_4=u_4^2.
\]
The action of $\lambda$ on $k(v_1,v_2,v_3,v_4)^{\langle \sigma\rangle}=k(w_1,w_2,w_3,w_4)$ is
\[
\lambda: w_1\mapsto \tfrac{w_1+w_1w_4+2w_2w_4}{(w_4-1)(w_1^2-w_2^2w_4)},~ w_2\mapsto -\tfrac{2w_1+w_2+w_2w_4}{(w_4-1)(w_1^2-w_2^2w_4)},~
w_3\mapsto \tfrac{1}{w_3w_4},~ w_4\mapsto w_4.
\]

\medskip
Step 3.
Define
\begin{gather*}
x_1=\frac{w_1+w_2w_4}{w_1+w_2},\quad x_2=\frac{w_2(w_4-1)(w_1^2-w_2^2w_4)}{(w_1+w_2)(2w_1+w_2+w_2w_4)}, \\
x_3=\frac{w_2w_3w_4(2w_1+w_2+w_2w_4)}{(w_1+w_2)^2}, \quad x_4=\frac{w_1^2-w_2^2w_4}{(w_1+w_2)^2}
\end{gather*}
then $K(x_1,x_2,x_3,x_4)=K(w_1,w_2,w_3,w_4)$ because
\[
w_1=\frac{x_2(x_1+x_4)}{x_4(x_1-1)},\quad w_2=-\frac{x_2(x_4-1)}{x_4(x_1-1)},\quad
w_3=\frac{x_3}{x_1^2-x_4},\quad w_4=-\frac{x_1^2-x_4}{x_4-1}.
\]
The action of $\lambda$ on $K(x_1,x_2,x_3,x_4)$ is given by
\[
\lambda: x_1\mapsto -x_1,~ x_2\mapsto \tfrac{x_4}{x_2},~ x_3\mapsto \tfrac{-(x_4-1)x_1^2+x_4(x_4-1)}{x_3},~ x_4\mapsto x_4.
\]

This is just the action given in the statement of this theorem.
\end{proof}

\begin{theorem} \label{t6.4}
Let $k$ be an infinite field with $\fn{char}k\ne 2$,
and $G=\langle\rho,\tau\rangle \simeq D_4$ act on the rational function field $k(x_1,x_2,x_3,x_4,x_5)$ by $k$-automorphisms defined as
\begin{align*}
\rho &: x_1\mapsto x_2,~ x_2\mapsto x_1,~ x_3\mapsto \tfrac{1}{x_1x_2x_3},~ x_4\mapsto x_5,~ x_5\mapsto \tfrac{1}{x_4}, \\
\tau &: x_1\mapsto x_3,~ x_2\mapsto \tfrac{1}{x_1x_2x_3},~ x_3\mapsto x_1,~ x_4\mapsto x_5,~ x_5\mapsto x_4.
\end{align*}

Then $k(x_1,x_2,x_3,x_4,x_5)^G$ is not retract $k$-rational.
In particular, it is not $k$-rational.
\end{theorem}

\begin{proof}

Step 1. By Lemma \ref{l3.2} we find
$k(x_1,x_2,x_3,x_4,x_5)^{\langle\rho^2\rangle}=k(X_1,\ldots,X_5)$
where
\[
X_1=x_1,\quad X_2=x_2,\quad X_3=x_3,\quad X_4:=\frac{x_4x_5+1}{x_4+x_5},\quad X_5:=\frac{x_4x_5-1}{x_4-x_5}.
\]
The actions of $\rho$ and $\tau$ on $K(X_1,\ldots,X_5)$ are given by
\begin{align*}
\rho &: X_1\mapsto X_2,~ X_2\mapsto X_1,~ X_3\mapsto \tfrac{1}{X_1X_2X_3},~ X_4\mapsto \tfrac{1}{X_4},~ X_5\mapsto -\tfrac{1}{X_5}, \\
\tau &: X_1\mapsto X_3,~ X_2\mapsto \tfrac{1}{X_1X_2X_3},~ X_3\mapsto X_1,~ X_4\mapsto X_4,~ X_5\mapsto -X_5.
\end{align*}

\medskip
Step 2.
We recall a result of invariant generators [\cite{HKY}, Lemma 3.9]:
Let $k$ be a field with $\fn{char}k\ne 2$, $c\in k\backslash \{0\}$, and $\tau$ act on $k(x,y,z)$ by
\[
\tau: x\leftrightarrow y,~ z\mapsto \tfrac{c}{xyz}.
\]

Then $k(x,y,z)^{\langle \tau \rangle}=k(t_1,t_2,t_3)$ where
\[
t_1=\frac{xy}{x+y},\quad t_2=\frac{xyz+\frac{c}{z}}{x+y},\quad t_3=\frac{xyz-\frac{c}{z}}{x-y}.
\]

\medskip
Step 3.
By the formula in Step 2,
we find that $k(X_1,\ldots,X_5)^{\langle \tau\rangle}=k(y_1,\ldots,y_5)$ where
\[
y_1=\frac{2X_1X_3}{X_1+X_3},~ y_2=\frac{X_1X_2X_3+\frac{1}{X_2}}{X_1+X_3},~
y_3=\frac{X_1X_2X_3-\frac{1}{X_2}}{X_1-X_3},~ y_4=X_4,~ y_5=\frac{2X_5}{X_1-X_3}.
\]
The action of $\rho$ on $k(y_1,\ldots,y_5)$ is given by
\[
\rho: y_1\mapsto -\tfrac{(y_2+y_3)(y_2-y_3)}{y_1y_2(y_3+1)(y_3-1)},~ y_2\mapsto \tfrac{1}{y_2},~ y_3\mapsto\tfrac{1}{y_3},~ y_4\mapsto \tfrac{1}{y_4},~ y_5\mapsto \tfrac{(y_2+y_3)(y_2-y_3)}{y_3y_5(y_2+1)(y_2-1)}.
\]

\medskip
Step 4.
Define
\begin{gather*}
z_1=\frac{y_3+1}{y_3-1},\quad z_2=\frac{\sqrt{-1}y_5(y_2-1)}{y_1y_2(y_3-1)}, \\
z_3=\frac{2\sqrt{-1} y_1y_2(y_3+1)}{(y_2+1)(y_3-1)},\quad
z_4=\frac{(y_2-1)(y_3+1)}{(y_2+1)(y_3-1)},\quad
z_5=\frac{(y_2-1)(y_4+1)}{(y_2+1)(y_4-1)},
\end{gather*}
then the action of $\rho$ on $k(y_1,\ldots,y_5)=k(z_1,\ldots,z_5)$ is given by
\[
\rho: z_1\mapsto -z_1,~z_2\mapsto\tfrac{z_4}{z_2},~ z_3\mapsto \tfrac{-(z_4-1)z_1^2+z_4(z_4-1)}{z_3},~ z_4\mapsto z_4,~ z_5\mapsto z_5.
\]

By Theorem \ref{t6.3}, $k(z_1,z_2,z_3,z_4)^{\langle \rho\rangle}$ is not retract $k$-rational.
Since $k(z_1,z_2,z_3,z_4,z_5)^{\langle\rho\rangle}$ $=k(z_1,z_2,z_3,z_4)^{\langle\rho\rangle}(z_5)$ and
the retract rationality is stable under rational extensions [\cite{SaA}, Proposition 3.6; \cite{Ka4}, Lemma 3.4],
it follows that $k(z_1,z_2,z_3,z_4,z_5)^{\langle \rho \rangle}$ is not retract $k$-rational.
\end{proof}

\begin{remark}
In Theorem \ref{t6.3}, by applying Yamasaki's result [\cite{Ya},
Lemma 4.3], we can show that
$k(x_1,x_2,x_3,x_4)^{\langle\tau\rangle}$ is not retract rational
over $k(x_4)$. But it is not clear that we can show that
$k(x_1,x_2,x_3,x_4)^{\langle\tau\rangle}$ is not retract
$k$-rational by this approach.

On the other hand, define
\[
t_1=x_1,~ t_2=-\frac{x_2(x_1^2-1)}{x_4-1},~ t_3=\frac{1}{2}\left(\!\frac{x_3}{x_4-1}-\frac{x_1^2-x_4}{x_3}\!\right),~
t_4=\frac{1}{2x_1}\left(\!\frac{x_3}{x_4-1}+\frac{x_1^2-x_4}{x_3}\!\right).
\]
Then $\tau$ acts on $K(t_1,t_2,t_3,t_4)=K(x_1,x_2,x_3,x_4)$ by
\[
\tau: t_1\mapsto -t_1,~ t_2\mapsto \tfrac{(t_1^2t_4^2-t_3^2+1)(t_1^2(t_4^2+1)-t_3^2)}{t_2},~ t_3\mapsto t_3,~ t_4\mapsto t_4.
\]
Define $t=t_1^2$. Then
$k(x_1,x_2,x_3,x_4)^{\langle\tau\rangle}=k(u,v,t,t_3,t_4)$ with
the relation
\[
u^2-tv^2=(tt_4^2-t_3^2+1)(t+tt_4^2-t_3^2).
\]
\end{remark}

\begin{theorem} \label{t6.5}
Let $G$ be a finite group, $G \simeq Gal(K/k)$, $M$ be a
decomposable $G$-lattice, i.e.\ $M=M_1\oplus M_2$ as
$\bm{Z}[G]$-modules where $1\le \fn{rank}_{\bm{Z}} M_1 <
\fn{rank}_{\bm{Z}} M$.

\rm{(1)} If $k$ is an infinite field, then $K(M)^G$ is retract
$k$-rational if and only if so are $K(M_1)^G$ and $K(M_2)^G$.

\rm{(2)} If $K(M_1)^G$ and $K(M_2)^G$ are $k$-rational, so is
$K(M)^G$.

\rm{(3)} If $\fn{rank}_{\bm{Z}} M_i \le 3$ for $i=1,2$, then
$K(M)^G$ is $k$-rational if and only if so are $K(M_1)^G$ and
$K(M_2)^G$.
\end{theorem}

\begin{proof}
Step 1.

First we show that, if $M$ decomposes as above, then $K(M)^G$ is
$k$-isomorphic to the free composite of $K(M_1)^G$ and $K(M_2)^G$
over $k$.

Let $T,T_1,T_2$ be the algebraic tori over $K$ whose character
modules are $M,M_1,M_2$ respectively. Note that the function
fields of $T,T_1,T_2$ are $K(M)^G, K(M_1)^G, K(M_2)^G$. Since the
category of character modules is anti-equivalent to the category
of algebraic tori [\cite{Vo2}, page 27, Example 6; \cite{KMRT},
page 333, Proposition 20.17], we find that $T$ is isomorphic to
$T_1 \times T_2$. Hence the result.

Alternatively, this result can be proved by showing $[K(M): L] \le
|G|$ where $L$ is the free composite of $K(M_1)^G$ and $K(M_2)^G$.
The details are omitted.

\medskip
Step 2.

Consider $K(M_1)^G \subset K(M)^G=K(M_1)(M_2)^G$. By [\cite {Sa3},
Theorem 1.3; \cite {Ka4}, Lemma 3.4 (iv)], $K(M_1)^G$ is a dense
retraction of $K(M)^G$. If $K(M)^G$ is retract $k$-rational, by
[\cite{Sa3}, Lemma 1.1; \cite{Ka4}, Lemma 3.4 (iii)], then
$K(M_1)^G$ is also retract $k$-rational. Similarly for $K(M_2)^G$.

Now assume that both $K(M_1)^G$ and $K(M_2)^G$ are retract
$k$-rational. Consider $k \subset K(M_1)^G \subset K(M)^G$. Since
$K(M_2)^G$ is retract $k$-rational, choose the affine domain $A$
whose quotient field is $K(M_2)^G$, the localized polynomial ring,
the $k$-algebraic morphisms, etc. provided in Definition
\ref{d6.1}. Tensor all these with $K(M_1)^G$. We find that the
free composite of $K(M_1)^G$ and $K(M_2)^G$ is retract rational
over $K(M_1)^G$. By Step 1, the free composite of $K(M_1)^G$ and
$K(M_2)^G$ is nothing but $K(M)^G$. We conclude that $K(M)^G$ is
retract rational over $K(M_1)^G$ and $K(M_1)^G$ is retract
rational over $k$. By [\cite{Ka4}, Theorem 4.2], $K(M)^G$ is
retract $k$-rational.

Alternatively, we may use Saltman's Theorem that $K(M)^G$ is
retract $k$-rational if and only if $[M]^{fl}$ is an invertible
$G$-lattice where $[M]^{fl}$ is the flabby class of $M$ (see
[\cite{Sa3}, Theorem 1.3] and [\cite{Ka4}, Section 2]). Note
that$[M_1\oplus M_2]^{fl}=[M_1]^{fl} \oplus [M_2]^{fl}$. The
details are omitted.

\medskip
Step 3.

If $K(M_1)^G$ and $K(M_2)^G$ are $k$-rational, $K(M)^G$ is also
$k$-rational by Step 1.

\medskip
Step 4.

We will prove that, if $K(M)^G$ is $k$-rational, then both
$K(M_1)^G$ and $K(M_2)^G$ are $k$-rational.

If $\fn{rank}_{\bm{Z}} M_i \le 2$, this follows from Theorem
\ref{t1.11}.

From now on, we assume that $\fn{rank}_{\bm{Z}} M_1 = 3$. Since
$K(M)^G$ is retract $k$-rational, we find that $K(M_1)^G$ is also
retract $k$-rational by Step 2.

In [\cite{Ku}, Theorem 1], Kunyavskii not only finds a birational
classification of $3$-dimensional algebraic tori, but also proves
that, if an algebraic torus is not $k$-rational, it is nor retract
$k$-rational (see Theorem \ref{t1.12} and [\cite{Ka4}, page 25,
the fifth paragraph]). Hence $K(M_1)^G$ is $k$-rational.
\end{proof}

\newpage
\renewcommand{\refname}{\centering{References}}


\begin{thebibliography}{BBNWZ}

\bibitem[AHK]{AHK}
H. Ahmad, M. Hajja and M. Kang, \textit{Rationality of some
projective linear actions}, J. Algebra 228 (2000), 643--658.


\bibitem[BBNWZ]{BBNWZ}
H. Brown, R. B\"ullow, J. Neub\"user, H. Wondratschek and H. Zassenhaus,
\textit{Crystallographic groups of four-dimensional spaces},
John Wiley, New York, 1978.

\bibitem[CHKK]{CHKK}
H. Chu, S.-J. Hu, M. Kang and B. E. Kunyavskii, \textit{Noether's
problem and the unramified Brauer group for groups of order 64},
International Math. Research Notices 12 (2010), 2329--2366.

\bibitem[CTS]{CTS}
J.-L. Colliot-Th\'el\`ene and J.-J. Sansuc, \textit{The
rationality problem for fields of invariants under linear
algebraic groups \rm{(}with special regards to the Brauer
groups\rm{)}}, in ``Proc. International Conference, Mumbai, 2004"
edited by V. Mehta, Narosa Publishing House, 2007.

\bibitem[Dr]{Dr}
P. K. Draxl, \textit{Skew fields}, London Math. Soc. Lecture Note
Series vol. 81, Cambridge Univ. Press, Cambridge, 1983.

\bibitem[GAP]{GAP}
The GAP Groups, GAP-Groups, Algorithms, and Programming, Version 4.4.10; 2007.
(http://www.gap-system.org)

\bibitem[Ha]{Ha}
M. Hajja,
\textit{Rationality of finite groups of monomial automorphisms of $K(x,y)$},
J. Algebra 109 (1987), 46--51.

\bibitem[HHR]{HHR}
K. Hashimoto, A. Hoshi and Y. Rikuna, \textit{Noether's problem
and $\bm{Q}$-generic polynomials for the normalizer of the 8-cycle
in $S_8$ and its subgroups}, Math. Comp. 77 (2008), 1153--1183.


\bibitem[HK1]{HK1}
M. Hajja and M. Kang,
\textit{Finite group actions on rational function fields}, J. Algebra 149 (1992), 139--154.

\bibitem[HK2]{HK2}
M. Hajja and M. Kang,
\textit{Three-dimensional purely monomial group actions},
J. Algebra 170 (1994), 805--860.

\bibitem[HKO]{HKO}
M. Hajja, M. Kang and J. Ohm,
\textit{Function fields of conics as invariant subfields},
J. Algebra 163 (1994), 383--403.

\bibitem[HKY]{HKY}
A. Hoshi, H. Kitayama and A. Yamasaki,
\textit{Rationality problem of three-dimensional monomial group actions},
J. Algebra 341 (2011), 45--108.

\bibitem[HR]{HR}
A. Hoshi and Y. Rikuna,
\textit{Rationality problem of three-dimensional purely monomial group actions: the last case},
Math. Comp. 77 (2008), 1823--1829.

\bibitem[Ka1]{Ka1}
M. Kang, \textit{Constructions of Brauer-Severi varieties and norm
hypersurfaces}, Canadian J. Math. 42 (1990), 230--238.

\bibitem[Ka2]{Ka}
M. Kang, \textit{Rationality problem of $GL_4$ group actions},
Advances in Math. 181 (2004), 321--352.

\bibitem[Ka3]{Ka2}
M. Kang, \textit{Some rationality problems revisited}, in
``Proceedings of the 4th ICCM, Hangzhou, 2007", edited by Lizhen
Ji, Kefeng Liu, Lo Yang and Shing-Tung Yau, Higher Education Press
(Beijing) and International Press (Somerville), 2007.

\bibitem[Ka4]{Ka3}
M. Kang, \textit{Retract rationality and Noether's problem},
International Math. Research Notices 15 (2009), 2760--2788.

\bibitem[Ka5]{Ka4}
M. Kang, \textit{Retract rational fields}, J. Algebra 349 (2012),
22--37.

\bibitem[Ki]{Ki}
I. Kiming, \textit{Explicit classifications of some $2$-extensions
of a field of characteristic different from $2$}, Canadian J.
Math. 42 (1990), 825--855.

\bibitem[KMRT]{KMRT}
M.-A. Knus, A. Merkurjev, M. Rost and J.-P. Tignol, \textit{The
book of involutions}, Amer. Math. Soc., Providence, 1998.

\bibitem[KMZ]{KMZ}
M. Kang, I. Michailov and J. Zhou, \textit{Noether's problem for
groups with a cyclic subgroup of index 4}, arXiv: 1108.3379.

\bibitem[KP]{KP}
M. Kang and B. Plans, \textit{Reduction theorems for Noether's
problem}, Proc. Amer. Math. Soc. 137 (2009), 1867--1874.

\bibitem[KPr]{KPr}
M. Kang and Y. G. Prokhorov, \textit{Rationality of
three-dimensional quotients by monomial actions}, J. Algebra 324
(2010), 2166--2197.


\bibitem[Ku]{Ku}
B. E. Kunyavskii, \textit{Three-dimensional algebraic tori},
Selecta Math. Soviet. 9 (1990), 1--21.

\bibitem[La]{La}
T. Y. Lam, \textit{The algebraic theory of quadratic forms}, W. A.
Benjamin, Inc., Reading, Massachusetts, 1973.

\bibitem[Le]{Le}
A. Ledet, \textit{Brauer type embedding problems}, Fields Inst.
Monographs vol. 21, Amer. Math. Soc., Providence, 2005.

\bibitem[MT]{MT}
Y. I. Manin and M. A. Tsfasman, \textit{Rational varieties:
algebra, geometry and arithmetric}, Russian Math. Survey 41
(1986), 51--116.

\bibitem[Ma]{Ma}
K. Masuda,
\textit{On a problem of Chevalley},
Nagoya Math. J. 8 (1955), 59--63.

\bibitem[Na]{Na}
M. Nagata, \textit{A theorem on valuation rings and its
applications}, Nagoya Math. J. 29 (1967), 85--91.

\bibitem[Oj]{Oj}
M. Ojanguren, \textit{The Witt group and the problem of L\"uroth},
ETS Editrice, Pisa, 1990.

\bibitem[Pr]{Pr}
Y. G. Prokhorov, \textit{Fields of invariants of finite linear
groups}, in ``Cohomological and geometric approaches to
rationality problems", edited by F. Bogomolov and Y. Tschinkel,
Progress in Math. vol. 282, Birkh\"auser, Boston, 2010.


\bibitem[Ro1]{Ro1}
P. Roquette, \textit{On the Galois cohomology of the projective
linear group and its applications to the construction of generic
splitting fields of algebras}, Math. Ann. 150 (1963), 411--439.

\bibitem[Ro2]{Ro2}
P. Roquette, \textit{Isomorphisms of generic splitting fields of
simple algebras}, J. Reine Angew. Math. 214/215 (1964), 207--226.

\bibitem[Sa1]{SaA}
D. J. Saltman, \textit{Retract rational fields and cyclic Galois
extensions}, Israel J. Math. 47 (1984), 165--215.

\bibitem[Sa2]{Sa3}
D. J. Saltman, \textit{Multiplicative field invariants}, J.
Algebra 106 (1987), 221--238.

\bibitem[Sa3]{Sa1}
D. J. Saltman, \textit{Twisted multiplicative field invariants,
Noether's problem and Galois extensions}, J. Algebra 131 (1990),
535--558.

\bibitem[Sa4]{Sa2}
D. J. Saltman, \textit{Multiplicative field invariants and the
Brauer group}, J. Algebra 133 (1990), 533--544.

\bibitem[Se]{Se}
J-P. Serre, \textit{Local fields}, Springer GTM vol. 67,
Springer-Verlag, Berlin, 1979.

\bibitem[Sw]{Sw}
R. G. Swan, \textit{Noether's problem in Galois theory}, in ``Emmy
Noether in Bryn Mawr", edited by B. Srinivasan and J. Sally,
Springer-Verlag, Berlin, 1983.

\bibitem[Vo1]{Vo1}
V. E. Voskresenskii, \textit{On two-dimensional algebraic tori II}
Math. USSR Izv. 1 (1967), 691--696.

\bibitem[Vo2]{Vo2}
V. E. Voskresenskii, \textit{Algebraic groups and their birational
invariants}, Transl. Math. Monographs vol. 179, Amer. Math. Soc.,
 Providence, 1998.


\bibitem[Ya]{Ya}
A. Yamasaki, \textit{Negative solutions to three-dimensional
monomial Noether problem}, arXiv:0909.0586.



\end{thebibliography}
\end{document}